\documentclass{amsart}

\usepackage{hyperref}
\usepackage{cleveref}
\usepackage{amssymb,amsfonts}
\usepackage[all,arc]{xy}
\usepackage{enumerate}
\usepackage{mathrsfs}
\usepackage{graphicx}
\usepackage{url}
\usepackage{amsmath}
\usepackage{amsthm}
\usepackage{stmaryrd}
\usepackage{arydshln}
\usepackage{xcolor}
\usepackage{chngcntr}
\usepackage{setspace}
\usepackage[usestackEOL]{stackengine}
\usepackage[margin=2.5cm]{geometry}
\usepackage[page,toc,titletoc,title]{appendix}

\usepackage{tikz-cd}
\usepackage{tikz}

\usetikzlibrary{arrows.meta,calc,patterns,positioning,shapes}
\usetikzlibrary{decorations.pathmorphing}
\tikzset{
	modal/.style={>=stealth,shorten >=1pt,shorten <=1pt,auto,
		node distance=1.5cm,semithick},
	world/.style={circle,draw,minimum size=1cm,fill=gray!15},
	point/.style={circle,draw,fill=black,inner sep=0.5mm},
	xpoint/.style={circle,draw,inner sep=0.5mm},
	reflexive/.style={->,in=120,out=60,loop,looseness=#1},
	reflexive/.default={5},
	reflexive point/.style={->,in=135,out=45,loop,looseness=#1},
	reflexive point/.default={25},
}

\DeclareRobustCommand{\squarearrow}{\mathrel{\tikz[baseline=-0.6ex]{\draw[-{Square[open]}, line width=0.3pt] (0,0) -- (0.35,0);}}}
\DeclareRobustCommand{\rbarbarrow}{\mathrel{\tikz[baseline=-0.6ex]{\draw[-{Arc Barb[reversed]}, line width=0.3pt] (0,-.085) -- (0,0.15);}}}

\definecolor{green}{rgb}{0.0,0.5,0.0}

\newtheorem{thm}{Theorem}[section]
\newtheorem{cor}[thm]{Corollary}
\newtheorem{prop}[thm]{Proposition}
\newtheorem{lem}[thm]{Lemma}

\theoremstyle{definition}
\newtheorem{defn}[thm]{Definition}

\newtheorem{fact}[thm]{Fact}

\theoremstyle{remark}

\newcommand{\lang}[1][\sigma]{\mathcal{L}(#1)}

\newcommand{\gridlang}[2]{\mathcal{L}_{#1,#2}}

\newcommand{\langdiv}{\,\vert\,}
\newcommand{\realth}[2]{\mathbb{S}_{#1,#2}}
\newcommand{\realthall}{\mathbb{S}}
\newcommand{\ur}[1]{#1^\urcorner}
\newcommand{\qo}[1]{#1_{\bar{x}}^{\bar{c}}}
\newcommand{\model}{M}

\newcommand{\nats}{\mathbb{N}}

\newcommand{\deno}[3]{\llbracket #1\rrbracket^{#2}_{#3}}

\newcommand{\setb}[2]{\{#1\,:\,#2\}}  
                  
\newcommand{\ps}[1]{\mathcal{P}(#1)}
\newcommand{\pso}[1]{\mathcal{P}_\omega(#1)}
\newcommand{\R}{\squarearrow}
\newcommand{\nd}[1]{\blacklozenge(#1)}
\newcommand{\yb}[1]{\boxtimes(#1)}

\newcommand{\K}{\mathsf{K}}
\newcommand{\fssfour}{\mathsf{FSS4}}

\newcommand{\sfive}{\mathsf{S5}}

\DeclareMathOperator{\dom}{dom}

\newcommand{\logic}{\mathsf{L}}

\newcommand{\class}[1][]{\mathcal{C}_{#1}}

\newcommand{\fs}{\mathsf{FS}}
\newcommand{\ipc}{\mathsf{IPC}}
\newcommand{\cpc}{\mathsf{CPC}}
\newcommand{\fofs}{\mathsf{FOFS}}
\newcommand{\fofsfour}{\mathsf{FOFS4}}
\newcommand{\fofsd}{\mathsf{FOFSD}}
\newcommand{\fofst}{\mathsf{FOFST}}
\newcommand{\fofsdfour}{\mathsf{FOFSD4}}
\newcommand{\fofssfour}{\mathsf{FOFSS4}}
\newcommand{\m}{\textsf{MIPC}}
\newcommand{\fofsfoura}{\mathsf{4}}
\newcommand{\fofsda}{\mathsf{D}}
\newcommand{\fofsta}{\mathsf{T}}
\newcommand{\fofsdfoura}{\mathsf{D4}}
\newcommand{\fofssfoura}{\mathsf{S4}}

\newcommand{\kba}{\mathbf{K \square a}}
\newcommand{\kbb}{\mathbf{K\square b}}
\newcommand{\kda}{\mathbf{K \lozenge a}}
\newcommand{\kdb}{\mathbf{K \lozenge b}}
\newcommand{\Univ}{\mathbf{UI}}
\newcommand{\Exist}{\mathbf{EG}}

\newcommand{\mopo}{\texttt{MP}}
\newcommand{\reg}{\texttt{Reg}}
\newcommand{\gen}{\texttt{Gen}}

\newcommand{\dboth}{\mathbf{D}}
\newcommand{\tb}{\mathbf{T \square}}
\newcommand{\td}{\mathbf{T \lozenge }}
\newcommand{\fourbox}{\mathbf{4\square}}
\newcommand{\fourdiamond}{\mathbf{4\lozenge}}

\newcommand{\fstwo}{\mathbf{FS2}}
\newcommand{\fsone}{\mathbf{FS1}}
\newcommand{\intuit}{\mathbf{I}}
\newcommand{\forallcon}{\mathbf{Con}}
\newcommand{\forallant}{\mathbf{Ant}}
\newcommand{\idref}{\mathbf{LI}}
\newcommand{\idsub}{\mathbf{InI}}
\newcommand{\NI}{\mathbf{NI}}
\newcommand{\ND}{\mathbf{ND}}

\newcommand{\I}{\preccurlyeq}
\newcommand{\Ires}[1][X]{\mathrel{{\preccurlyeq}|_#1}}
\newcommand{\M}{\mathrel{R}}
\newcommand{\Mres}[1][X]{\mathrel{R|_#1}}

\newcommand{\U}{\rbarbarrow}
\newcommand{\D}{D}
\newcommand{\DW}{d}

\newcommand{\DE}[1][-]{\asymp^{#1}}
\newcommand{\In}{I}
\newcommand{\Inn}{I^\nu}
\newcommand{\Inp}{I^\pi}

\newcommand{\WC}{\mathfrak{W}}
\newcommand{\IC}{\preccurlyeq_\mathfrak{t}}
\newcommand{\MC}{\mathrel{\mathfrak{R}}}
\newcommand{\DC}{\mathfrak{D}}
\newcommand{\DWC}{\mathfrak{d}}

\newcommand{\DEC}[1][-]{\asymp_\mathfrak{t}^{#1}}
\newcommand{\InC}{\mathfrak{I}}
\newcommand{\InnC}{\mathfrak{I}^\nu}
\newcommand{\InpC}{\mathfrak{I}^\pi}

\newcommand{\trf}{\mathfrak{F}}
\newcommand{\trfp}[2]{\trf_{#1,#2}}
\newcommand{\trm}{\mathfrak{M}}
\newcommand{\trmp}[2]{\trm_{#1,#2}}

\newcommand{\equals}{\approx}

\let\oldepsilon\epsilon
\let\epsilon\varepsilon
\let\varepsilon\oldepsilon

\let\oldphi\phi
\let\phi\varphi
\let\varphi\oldphi

\let\bigland\bigwedge
\let\biglor\bigvee

\crefname{subsection}{subsection}{subsections}
\newcommand{\DFref}[1]{Lemma \hyperref[#1]{\ref*{theorems}(\ref*{#1})}}
\newcommand{\doi}[1]{\url{https://doi.org/#1}}

\title{First-Order Fischer Servi Logic}

\author{A. Christensen}

\begin{document}
	
\maketitle

\begin{abstract}
	We prove the completeness of a first-order analogue of the Fischer Servi logic $\fs$ with respect to its expected birelational semantics. To this end we introduce the notion of the \textit{trace model} and, much like in a canonical model argument, prove a truth lemma. We conclude by examining a number of other first-order Fischer Servi logics, including a first-order analogue of $\fssfour$, whose completeness can be similarly proved.
\end{abstract}

\section{Introduction}
Taking the classical propositional calculus ($\cpc$) as a starting point, the majority of studied logics are obtained by some combination of adding quantifiers, weakening the propositional calculus, and introducing modalities. Not only does each of these ways of moving away from $\cpc$ have an extensive literature, but the pairwise combinations have also attracted considerable attention. Starting with quantified non-classical logics, first-order intuitionistic logic, formalized by Heyting \cite{heyting}, was born out of Brouwer's intuitionism \cite{brouwer} and accordingly is part of a rich tradition in the philosophy of mathematics. Quantified modal logic, first introduced to the literature by Barcan \cite{barcan} (see \cite{conandjan} for historical context), received early treatment from Kripke \cite{kripkesemcon}, while Quine \cite{quinequant,quineref} brought some of its apparent paradoxes into focus. Lastly, non-classical modal logics have been studied from the perspectives of mathematical logic (e.g., \cite{goldblatt,wolandzak}), theoretical computer science (e.g., \cite{bentonetal,bieanddep,depandrit}), and philosophy (e.g., \cite{field,holandman}).

Recently, there has been an uptick in interest in a particular non-classical modal logic: the Fischer Servi logic ($\fs$), which is an intuitionistic analogue of $\K$, the minimal normal modal logic on a classical base. The finite model property for $\fssfour$, which was an open problem introduced in Simpson's thesis \cite{simpson}, was resolved in the affirmative by Girlando et al. \cite{girlando}. In the domain of philosophical logic, Bobzien and Rumfitt \cite{bobzienandrumfitt} proposed that some extension of $\fs$ could be suitable for an intuitionistic modal logic of vagueness. There are formal reasons for privileging $\fs$, as well. Fischer Servi provides an argument that $\fs$ is the correct analogue of $\K$ by proving that it fully and faithfully embeds into a certain bimodal logic \cite{fischerservi2,fischerservi}. Perhaps more convincingly, Simpson \cite{simpson} notes that Stirling proved that the standard translation of modal formulas into first-order formulas (see, e.g., \cite[p.~84]{blackburn}) fully and faithfully embeds $\fs$ into first-order intuitionistic logic, providing additional evidence for its naturality.

In spite of the aforementioned contributions, however, very little work has been done generalizing $\cpc$ in all three directions simultaneously.\footnote{An important exception to this is quantified lax logic, most notably treated by Goldblatt in \cite{goldblatt2}.} In this paper, we will focus our attention on a first-order intuitionistic analogue of $\K$, which in some sense is the most obvious way to marry the three modifications---first-order logic is by far the most studied quantified logic, intuitionistic logic is the non-classical logic with the most historical precedent, and $\K$ enjoys a special status as the minimal normal modal logic. In this area, there is one positive result by Wijesekera \cite{wijesekera} showing completeness of a certain first-order intuitionistic modal logic with respect to a relational semantics. Wijesekera's frames are minimally structured in the sense that the intuitionistic relation is a partial order, the modal relation is unconstrained, and there are no compatibility conditions between the two. Even in this ostensibly simple case, Wijesekera needs to use a construction more delicate than a standard canonical model to establish the desired model existence theorem. Later, Gao \cite{gao} claimed a completeness proof for a first-order extension of $\mathsf{MIPC}$,\footnote{$\mathsf{MIPC}$ is an $\sfive$-like intuitionistic modal logic first introduced by Prior \cite{prior} and later studied by Bull \cite{bull} and Fischer Servi \cite{fischerservi2}.} though at least one crucial lemma was false, rendering the proof apparently unsalvageable \cite{bimbo}.

Our aim is to establish completeness of first-order Fischer Servi logic with respect to the expected birelational semantics. To this end, we introduce the concept of a trace model. After this, we note that the trace model construction can be easily adapted to establish completeness for five other first-order Fischer Servi logics, including the first-order analogue of $\fssfour$, the subject of \cite{girlando}. Additionally, for any of the six logics, we will treat three further extensions given by considering the necessity of identity and the necessity of distinctness.

The outline of this paper is as follows. \Cref{thelogic} introduces a Hilbert a system for first-order Fischer Servi logic and establishes a slew of deductive facts that are needed for the completeness proof. After this, a relational semantics is presented in \Cref{semantics}. \Cref{trace} then constructs the trace model and establishes a truth lemma, thereby establishing completeness. Finally, some extensions of the logic are treated in \Cref{otherlogics}.

\section{The Logic $\fofs$}
\label{thelogic}
\subsection{The Language}
We will work with a class of languages parametrized by signatures, which keep track of the non-logical symbols in the usual way.
\begin{defn}
	A \textit{signature} $\sigma=(\sigma^\nu,\sigma^\pi,\sigma^\alpha)$ is a triple where $\sigma^\nu$ is a countably infinite set of constants, $\sigma^\pi$ is a set countably many predicate symbols, and $\sigma^\alpha:\sigma^\pi\to\nats_{>0}$ assigns arities to the predicate symbols. 
\end{defn}
The cardinality requirements can be dispensed with, but we choose to keep the signature countable in order to streamline some arguments that would otherwise need to appeal to transfinite recursion. Additionally, the requirement that we have infinitely many constants is due to the fact that the Hilbert system we will introduce produces only sentences as theorems. This requirement could be removed by tweaking the system so that formulas with free variables could count as theorems. Functions could also be easily accommodated if desired.

Fix a countably infinite set of variables $V=\{v_0,v_1,v_2,\ldots\}$. Given a signature $\sigma$, we define the language $\mathcal{F}(\sigma)$ recursively: 
\[\phi::= P(t_1,\ldots,t_n)\langdiv s\equals t \langdiv(\phi\land\phi)\langdiv \langdiv (\phi \lor \phi) \langdiv (\phi\to\phi) \langdiv \square \phi\langdiv \lozenge\phi\langdiv \forall x \phi\langdiv \exists x \phi \langdiv \bot\]
where $x\in V$, $P\in\sigma^\pi$, $\sigma^\alpha(P)=n$, and $s$, $t$, and the $t_i$s are terms. We will write $\neg\phi$ for $(\phi\to\bot)$, $(\phi\leftrightarrow\psi)$ for $((\phi\to\psi)\land(\psi\to\phi))$, and $\top$ for $\neg\bot$. If $\phi$ is a formula and $c$ is a constant, $\phi_x^c$ is the result of replacing each instance of $c$ with an instance of $x$. If $\bar{c}=\langle c_1,\ldots, c_p\rangle$ is a list of distinct constants and $\bar{x}=\langle x_1,\ldots, x_p\rangle$ is a list a of distinct variables, we will write $\qo{\phi}$ for $(\cdots((\phi_{x_1}^{c_1})_{x_2}^{c_2})\cdots)_{x_n}^{c_n}$.

A \textit{sentence} is a formula with no free variables, and we write $\lang$ for the set of sentences in $\mathcal{F}(\sigma)$. A \textit{one-place formula} is a formula with exactly one free variable. If $\phi$ is a one-place formula, we might write $\phi(t)$ for the result of replacing every instance of the free variable with a term $t$. Note that because we will use this notation for one-place formulas exclusively, $\phi(c)$ is a sentence for any constant $c$.

\subsection{The Hilbert System}
A Hilbert system for the logic $\fofs$ is presented in \Cref{hilbertsystem}. Naturally, the system is simply a collection of axioms for $\fs$ together with axioms for first-order intuitionistic logic. The only real modification is that indiscernibility of identicals ($\idsub$) is restricted to non-modal formulas. \Cref{idextensions} handles the case where this restriction is lifted.

\begin{figure}[h]
	\begin{center}
		\begin{tabular}{| l | l | l |}
			\hline
			\multicolumn{3}{| c |}{The Logic $\fofs$}\\
			\hline
			$\intuit$  &  any substitution instance of a theorem of $\ipc$&\\
			$\kba$ & $\square (\phi\land \psi)\leftrightarrow (\square \phi\land\square \psi)$&\\
			$\kbb$ & $\square \top$&\\
			$\kda$ & $\lozenge (\phi\lor \psi)\leftrightarrow (\lozenge \phi\lor\lozenge \psi)$&\\
			$\kdb$ & $\neg\lozenge\bot$&\\
			$\fsone$ & $(\lozenge\phi\to\square\psi)\to\square(\phi\to\psi)$&\\
			$\fstwo$ & $\lozenge(\phi\to\psi)\to (\square\phi\to\lozenge\psi)$&\\
			$\Univ$ & $\forall x\phi(x)\to\phi(c)$&\\
			$\Exist$ & $\phi(c)\to\exists x\phi(x)$&\\
			$\forallant$ & $\forall x(\phi(x)\to\psi)\to (\exists x\phi(x)\to \psi)$& $\phi$ is a one-place formula\\
			$\forallcon$ & $\forall x(\phi\to\psi(x))\to (\phi\to \forall x\psi(x))$& $\psi$ is a one-place formula\\
			$\idref$ & $c\equals c$& \\
			$\idsub$ & $c_1 \equals c_2 \to (\phi(c_1)\to\phi(c_2))$& $\phi$ is modal-free\\
			\hline
			$\mopo$ & from $\phi\to \psi$ and $\phi$ infer $\psi$&\\
			$\gen$ & from $\phi(c)$ infer $\forall x\phi(x)$& $x$ is substitutable for $c$ in $\phi$\\
			$\reg$ & from $\phi\to\psi$ infer $\bigcirc\phi\to\bigcirc\psi$& $\bigcirc\in\{\square,\lozenge\}$\\
			\hline
		\end{tabular}
	\end{center}
	\caption{$c,c_1,c_2\in \sigma^\nu$, $x\in V$, and $\phi$ and $\psi$ are sentences unless otherwise noted}
	\label{hilbertsystem}
\end{figure}

We write $\vdash\phi$ if $\phi$ is a theorem of $\fofs$. For some $\Gamma\subseteq \lang$, we write $\Gamma\vdash \phi$ if there exists a finite set of formulas $\{\gamma_i\}_i$ such that $\vdash \bigland_i \gamma_i \to \phi$. Consistency of a pair of sets of sentences can then be defined.

\begin{defn}
	Where $\Gamma,\Omega\subseteq\lang$, we say that the pair $(\Gamma,\Omega)$ is \emph{consistent} if there does not exist a finite set of formulas $\{\gamma_i\}_i\subseteq \Omega$ such that $\Gamma\vdash\biglor_i\gamma_i$.
\end{defn}

Our next order of business is to introduce saturated theories, the natural generalizations of prime theories to the first-order setting. 

\begin{defn}
	A set of formulas $\Gamma\subseteq \lang$ is called a \emph{saturated theory} if it satisfies the following:
	\begin{center}
		\begin{tabular}{ll}
			(Consistency)& $\Gamma\nvdash\bot$.\\
			(Deductive closure)& For any $\phi\in\lang$, if $\Gamma\vdash\phi$, then $\phi\in \Gamma$.\\
			(Primeness)& If $\phi\lor\psi\in \Gamma$, then $\phi\in\Gamma$ or $\psi\in\Gamma$.\\
			(Henkinness)& If $\exists x\phi(x)\in\Gamma$, then $\phi(c)\in \Gamma$ for some constant $c\in\sigma^\nu$.
		\end{tabular}
	\end{center}
\end{defn}

For a set of sentences $\Gamma\subseteq \lang$, we write $\yb{\Gamma}$ for $\setb{\phi\in\lang}{\square\phi\in \Gamma}$. If additionally $\Gamma$ is a saturated theory, we write $\nd{\Gamma}$ for $\setb{\phi\in\lang}{\lozenge\phi\notin \Gamma}$. In later sections we will have many languages (varying only by constants) in play. We will often write $\nd{\Gamma}$ where $\Gamma$ is a saturated theory in one of them. Note that there is no ambiguity here since we can read off the constants of the language from $\Gamma$: a constant $c$ will be in the signature if and only if $c\equals c\in\Gamma$. The next lemma will be used implicitly throughout the paper. Its purpose is largely to allow us to use cleaner notation, replacing conjunctions or disjunctions with single sentences.

\begin{lem}
	Let $\Gamma$ be saturated. Then $\yb{\Gamma}$ is closed under conjunction and $\nd{\Gamma}$ is closed under disjunction.
\end{lem}
\begin{proof}
	Closure under conjunction for $\yb{\Gamma}$ follows from $\kba$ and the deductive closure of $\Gamma$. Closure under disjunction for $\nd{\Gamma}$ follows the primeness of $\Gamma$. 
\end{proof}

We are now in a position to provide a key extension result. The proof is completely standard in the first-order intuitionistic setting (see, e.g., \cite{dragalin,wijesekera}), but some of the argument's ingredients will appear in a later proof of another extension result.

\begin{lem}
	\label{disjunctionstep}
	Suppose that $(\Gamma,\Omega)$ is a consistent pair and that $\Gamma\vdash\phi\lor\psi$. If $(\Gamma\cup\{\phi\},\Omega)$ is inconsistent, then $(\Gamma\cup\{\psi\},\Omega)$ is consistent.
\end{lem}
\begin{proof}
	Suppose that both $(\Gamma\cup\{\phi\},\Omega)$ and $(\Gamma\cup\{\psi\},\Omega)$ are inconsistent. Then we have $\Gamma\vdash \phi \to \beta$ and $\Gamma\vdash \psi\to \beta'$ where $\beta$ and $\beta'$ are disjunctions of elements of $\Omega$. Since $\Gamma\vdash \phi\lor\psi$, intuitionistic reasoning grants us $\Gamma\vdash\beta\lor\beta'$, contradicting the consistency of $(\Gamma,\Omega)$.
\end{proof}

\begin{lem}
	\label{existentialstep}
	Suppose that $(\Gamma,\Omega)$ is a consistent pair and that $\Gamma\vdash\exists x\phi(x)$. If $c$ is a constant not appearing in a formula in $\Gamma$ or $\Omega$, then $(\Gamma\cup\{\phi(c)\},\Omega)$ is consistent.
\end{lem}
\begin{proof}
	Suppose that $\Gamma\vdash \exists x \phi(x)$ but that $(\Gamma\cup\{\phi(c)\},\Omega)$ is inconsistent. Then $\vdash\alpha\land \phi(c)\to\beta$ where $\alpha$ is a conjunction of elements of $\Gamma$ and $\beta$ is a disjunction of elements of $\Omega$. This is equivalent to $\vdash\alpha\to(\phi(c)\to\beta)$. By $\gen$, $\vdash \forall x(\alpha\to(\phi(x)\to\beta))$. Applying $\forallcon$, we have $\vdash\alpha\to\forall x(\phi(c)\to\beta)$, so $\Gamma\vdash\forall x(\phi(c)\to\beta)$. By $\forallant$, $\Gamma\vdash \exists x\phi(x)\to\beta$, so $\Gamma\vdash \beta$, contradicting that $(\Gamma,\Omega)$ is consistent.
\end{proof}

\begin{prop}
	\label{getreal}
	Suppose that $(\Gamma,\Omega)$ is a consistent pair and that $\sigma'$ is some signature such that $\Gamma,\Omega\subseteq \lang[\sigma']$ and $\sigma'^\nu$ contains infinitely many constants that do not appear in a formula in $\Gamma$ or $\Omega$. Then there exists a saturated theory $\Gamma'$ over $\lang[\sigma']$ such that $\Gamma\subseteq \Gamma'$ and $\Omega\cap\Gamma'=\varnothing$.
\end{prop}
\begin{proof}
	Fix the following enumerations in $\lang[\sigma']$:
	\begin{align*}
		\delta_i &= \phi^0_i\lor\phi^1_i&\text{for all disjunctive sentences}\\
		\epsilon_i &= \exists x\psi_i(x)&\text{for all existentially quantified sentences.}
	\end{align*}
	Set $\Gamma_0 = \Gamma$.
	\begin{itemize}
		\item If $n+1$ is odd, take the least index $k$ such that $\Gamma\vdash \delta_k$ but $\Gamma_n$ proves neither $\phi_k^0$ nor $\phi_k^1$. (If none exists, set $\Gamma_{n+1}=\Gamma_n$.) Set \[\Gamma_{n+1}=\begin{cases}
			\Gamma_n\cup\{\phi_k^0\} & (\Gamma_n\cup\{\phi_k^0\}, \Omega) \text{ is consistent} \\
			\Gamma_n\cup\{\phi_k^1\} & \text{else.}
		\end{cases}\]
		\item If $n+1$ is even, let $k$ be the least index such that $\Gamma$ proves $\epsilon_k$ but does not prove $\psi_k(c)$ for any constant $c$. (If none exists, set $\Gamma_{n+1}=\Gamma_n$.) Set $\Gamma_{n+1} =\Gamma_n\cup\{ \phi(c')\}$, where $c'$ is some constant in $\sigma'^\nu$ not appearing in $\Gamma_n$.
	\end{itemize}
	Let $\Gamma'$ be the deductive closure of  $\bigcup_{n<\omega} \Gamma_n$.
	
	First we check that $(\Gamma',\Omega)$ is consistent. It is sufficient to verify by induction that for each $n$, $(\Gamma_n,\Omega)$ is consistent, since every consequence of $\Gamma'$ is a consequence of one of the $\Gamma_n$s. If $n=0$, this is just by assumption. Otherwise if $n$ is odd, we have two cases. If $(\Gamma_n\cup\phi_k^0,\Omega)$ is consistent, there is nothing to check. Otherwise, we apply the induction hypothesis and \Cref{disjunctionstep}. For the case where $n$ is even, we are done by \Cref{existentialstep}.
	
	To see that $\Gamma'$ is prime, suppose that $\phi^0\lor\phi^1\in \Gamma'$. This means that $\Gamma_n\vdash\phi\lor\psi$ for some $n$. By the construction, either $\phi^0$ or $\phi^1$ was added to our set some finite number of steps later. Henkinness is checked similarly, since if $\Gamma'\vdash\exists x\psi(x)$, then for some $n$, $\Gamma_n\vdash \exists x\psi(x)$. A witness was then added some finite number of steps later. 
\end{proof}

\subsection{Useful Deductive Facts}
We conclude this section by presenting a slue of useful theorems and rules derivable in the system $\fofs$. We start with the theorems.

\begin{lem} The following are theorems of $\fofs$:
	\label{theorems}
	\begin{enumerate}[{\normalfont (1)}]
		\item \label{necdis}$\biglor_i\square\phi_i\to \square\biglor_i\phi_i$,
		\item $(\lozenge\phi\land\square\psi)\to \lozenge(\phi\land \psi)$,
		\item \label{fstworeverse} $(\lozenge\phi\land \square(\phi\to\psi))\to\lozenge\psi$,
		\item \label{eucid} $(c_1\equals c_2\land c_1\equals c_3)\to (c_2\equals c_3)$,
		\item \label{barcanish} $\lozenge\forall x\phi(x)\to\forall x\lozenge\phi(x)$.
	\end{enumerate}
\end{lem}
\begin{proof}\hphantom{.}
	\begin{enumerate}
		\item For each $k$, we have $\vdash \phi_k\to\biglor_i\phi_i$ by intuitionistic reasoning. Regularity of the box implies $\vdash\square\phi_k\to\square\biglor_i \phi_i$. More intuitionistic reasoning allows us to conclude $\vdash \biglor_i\square\phi_i\to \square\biglor_i\phi_i$.
		\item By $\intuit$, we have $\vdash \phi\to(\psi\to(\phi\land\psi))$. Applying $\reg$, $\vdash \lozenge\phi\to\lozenge(\psi\to(\phi\land\psi))$. An instance of $\fstwo$ is $\lozenge(\psi\to(\phi\land\psi))\to (\square\psi\to\lozenge(\phi\land \psi))$, so we have $\vdash\lozenge\phi\to(\square\psi\to\lozenge(\phi\land \psi))$, which is intuitionistically equivalent to the desired theorem. 
		\item By the previous fact, $\vdash (\lozenge\phi\land \square(\phi\to\psi))\to\lozenge(\phi\land(\phi\to\psi))$. Additionally, $\intuit$ and $\reg$ entail $\vdash \lozenge(\phi\land (\phi\to\psi))\to\lozenge \psi$. Therefore, $\vdash (\lozenge\phi\land \square(\phi\to\psi))\to\lozenge\psi$.
		\item An instance of $\idsub$ where $\phi(x)=x\equals c_3$ yields $\vdash c_1\equals c_2\to(c_1\equals c_3\to c_2\equals c_3)$, which is intuitionistically equivalent to the claimed theorem.
		\item Let $c$ be some constant not appearing in $\phi$. By $\Univ$, $\vdash \forall x\phi(x)\to\phi(c)$. Applying $\reg$ for $\lozenge$, we have $\vdash\lozenge\forall x\phi(x)\to\lozenge\phi(c)$. Since $c$ does not appear in $\phi$, $\gen$ affords us $\vdash\forall x (\lozenge\forall x\phi(x)\to\lozenge\phi(x))$. Finally, $\forallcon$ yields $\vdash \lozenge\forall x\phi(x)\to\forall x\lozenge\phi(x)$.\qedhere
	\end{enumerate}
\end{proof}

The rule of necessitation is also derivable.

\begin{lem}[Necessitation]
	\label{nec}
	If $\vdash\phi$, then $\vdash\square\phi$.
\end{lem}
\begin{proof}
	If $\vdash \phi$, intuitionistic reasoning grants us $\vdash\top\to\phi$. Regularity of the box implies $\vdash\square\top\to\square\phi$. Finally, $\kba$ and $\mopo$ allow us to conclude $\vdash\square\phi$.
\end{proof}

We will not need to use necessitation directly, however. Instead, we will make use of the following lemma, which will often be referenced by simply mentioning that we are \textit{reasoning under the box}.

\begin{cor}[Reasoning under the box]
	\label{rubox}
	If $\Gamma\vdash \square\phi$, $\Gamma\vdash \square \psi$, and $\vdash (\phi\land\psi)\to\chi$, then $\Gamma\vdash \square\chi$.
\end{cor}
\begin{proof}
	First note that $\vdash ((\phi\land\psi)\land ((\phi\land\psi)\to\chi))\to \chi$ by $\intuit$. By the regularity of $\square$, we have $\vdash \square((\phi\land\psi)\land ((\phi\land\psi)\to\chi))\to \square\chi$. By $\kba$, this is equivalent to $\vdash (\square(\phi\land\psi)\land \square((\phi\land\psi)\to\chi))\to \square\chi$. By the first assumption and $\kba$, $\Gamma\vdash \square(\phi\land\psi)$. By the second assumption and \Cref{nec}, $\vdash\square ((\phi\land\psi)\to \chi)$. Therefore, $\Gamma\vdash\square\chi$.
\end{proof}

Finally we have two lemmas that exhibit useful deductive moves from first-order intuitionistic logic.

\begin{lem}
	\label{gengen}
	Suppose that the constant $c$ appears in neither any of the sentences in $\Gamma\subseteq\lang$ nor in $\phi(x)$. If $\Gamma\vdash \phi(c)$, then $\Gamma\vdash\forall x\phi(x)$. 
\end{lem}
\begin{proof}
	For some finite collection of sentences $\{\gamma_i\}_i\subseteq \Gamma$, we have $\vdash\bigland_i \gamma_i \to\phi(c)$. By $\gen$, $\vdash \forall x(\bigland_i \psi_i \to\phi(x))$, and then by $\forallcon$, $\vdash \bigland_i \psi_i \to\forall x\phi(x)$. Therefore $\Gamma\vdash\forall x\phi(x)$.
\end{proof}

\begin{lem}
	\label{exsuff}
	Suppose that $\Gamma\vdash \phi\to\psi$ and that $\phi$ contains constants $\bar{c}=\langle c_1,\ldots,c_p\rangle$ not appearing in $\psi$ or any sentences of $\Gamma$. Then $\Gamma\vdash \exists \bar{x}\qo{\phi}\to\psi$ where each variable $x_i$ in $\bar{x}=\langle x_1,\ldots x_p\rangle$ does not appear in $\phi$.
\end{lem}
\begin{proof}
	First we apply \Cref{gengen} to obtain $\Gamma\vdash \forall \bar{x}(\qo{\phi}\to\psi)$. Repeated application of $\forallant$ yields $\Gamma\vdash \exists\bar{x}\qo{\phi}\to\psi$.
\end{proof}

\section{Relational Semantics}
\subsection{Frames}
\label{semantics}
We begin by recalling the definition of a Fisher Servi frame, which are the natural relational structures for semantics for $\fs$. 
\begin{defn}
	\label{fsframe}
	A \textit{Fischer Servi frame} is a triple $(W,\preccurlyeq, \M)$ where $\preccurlyeq$ is a partial order and $\M$ is a binary relation that satisfy the following compatibility conditions:
	\begin{enumerate}[(FC1)]
		\item $(\preccurlyeq\circ\M)\subseteq(\M\circ \preccurlyeq) $;
		\item $(\preccurlyeq\circ \M^{-1})\subseteq (\M^{-1}\circ\preccurlyeq)$.
	\end{enumerate}
\end{defn}

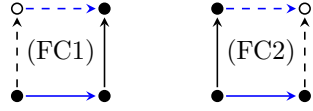
\begin{figure}
	\begin{center}
		\begin{tikzpicture}[modal]
			\node[point] (w) {};
			\node[xpoint] (w') [above=1cm of w] {};
			\node[point] (x) [right=1cm of w] {};
			\node[point] (x') [right=1cm of w'] {};
			\path[->] (w) edge[color=blue] (x);
			\path[->] (x) edge (x');
			\path[->] (w) edge[dashed] (w');
			\path[->] (w') edge[color=blue,dashed] (x');
			\node (one) at ($(w)!0.5!(x')$) {(FC1)};
			\node[point] (W) [right=2.5cm of w] {};
			\node[point] (W') [above=1cm of W] {};
			\node[point] (X) [right=1cm of W] {};
			\node[xpoint] (X') [right=1cm of W'] {};
			\path[->] (W) edge[color=blue] (X);
			\path[->] (X) edge[dashed] (X');
			\path[->] (W) edge (W');
			\path[->] (W') edge[color=blue,dashed] (X');
			\node (two) at ($(W)!0.5!(X')$) {(FC2)};
		\end{tikzpicture}
	\end{center}
	\caption{The Fischer Servi frame conditions from \Cref{fsframe} where blue arrows indicate the modal relation, black arrows the intuitionistic relation, and dotted arrows and hollow points existence claims}
\end{figure}

Frames for $\fofs$ are then obtained by attaching a domain and an equivalence relation to each point in a Fischer Servi frame in the following way.

\begin{defn}
	A \textit{domain system} $\D$ on a Fischer Servi frame $(W,\I,\M)$ is a pair $\D=(\DW,\DE)$ such that the following hold:
	\begin{enumerate}[(i)]
		\item $\DW$ assigns to each $w\in W$ a set $\DW(w)$ such that the following hold:
		\begin{itemize}
			\item If $w\I w'$, $\DW(w)\subseteq \DW(w')$;
			\item If $w\M v$, $\DW(w)\subseteq \DW(v)$.
		\end{itemize}
		\item $\DE$ assigns to each $w\in W$ an equivalence relation $\DE[w]$ such that if $a_1\DE[w] a_2$ and $w\I w'$, then $a_1\DE[w'] a_2$.
	\end{enumerate}
\end{defn}

\begin{defn}
	An \textit{$\fofs$ frame} is a 4-tuple $(W, \I, \M, \D)$ where $(W,\I,\M)$ is a Fischer Servi frame on which $\D$ is a domain system.
\end{defn}

We can immediately establish the convenient fact that certain restrictions of $\fofs$ frames are again $\fofs$ frames.

\begin{lem}
	\label{frres}
	Let $F=(W, \I, \M, \D)$ be an $\fofs$ frame. Suppose that $X\subseteq W$ is closed under $\I$ and $\M$. Then $F|_X=(X,\Ires,\Mres,\D|_X)$ is a frame.
\end{lem}
\begin{proof}
	First we check that $F|_X=(X,\Ires,\Mres)$ is a Fischer Servi frame. To verify (FC1), we assume $w\Mres v$ and $v\Ires v'$. This means that $w\M v$ and $v\I v'$ in $F$. Because $(W, \I, \M)$ is a Fischer Servi frame there must be a $w'\in W$ such that $w\Ires w'$ and $w'\Mres v'$, and since $X$ is closed under $\I$, $w'\in X$, and (FC1) holds. The argument for (FC2) is similar. If $w\Ires w'$ and $w \Mres v$, we in fact have $w\I w'$ and $w\M v$. There is then a $v'$ such that $v\I v'$ and $w'\M v'$. Since $X$ is closed under $\M$ and $w'\in X$, we have $v'\in X$ as well, so (FC2) holds.
	
	It is immediate that each of the conditions on a domain system holds after restriction. (In fact, this does not even depend on the closure assumption on $X$.)
\end{proof}

\subsection{Models}

\begin{defn}
	\label{inter}
	An \textit{interpretation} $\In$ on an $\fofs$ frame $(W, \I, \M, \D)$ is a pair $\In=(\Inn,\Inp)$ such that the following hold (where $\D=(\DW,\DE)$):
	\begin{enumerate}[(i)]
		\item $\Inn$ assigns to each $c\in \sigma^\nu$ a domain element $d$ such that for every $w$, $d\in\DW(w)$.
		\item $\Inp$ assigns to each $w\in W$ a function $\Inp(w):\sigma^\pi\to \pso{D(w)}$ such that:
		\begin{itemize}
			\item If $\sigma^\alpha(P)=n$, then $\Inp(w)(P)\subseteq \ps{D(w)^n}$;
			\item If $w\I w'$, $\Inp(w)(P)\subseteq \Inp(w')(P)$;
			\item If $a_i\DE[w] b_i$ for each $1\leq i \leq n$ and $\bar{a}\in \Inp(w)(P)$, then $\bar{b}\in\Inp(w)(P)$.
		\end{itemize}
	\end{enumerate}
\end{defn}

\begin{defn}
	An \textit{$\fofs$ model} is a 5-tuple $(W,\I,\M,\D, \In)$ where $(W,\I,\M,\D)$ is an $\fofs$ frame on which $\In$ is an interpretation.
\end{defn}
 
We choose to give our first-order semantics via variable assignment functions, which are fairly easy to use given our expanding domain assumptions on domain systems. 
 
\begin{defn}
	If $(W,\I,\M,\D, \In)$ is an $\fofs$ model and $w\in W$, a \textit{$w$-variable assignment} ($w$-VA) is a function $g:V\to \DW(w)$.
\end{defn}
Where $g$ is a VA, we will write $g[x:=a]$ for some $a$ to mean the new VA obtained by changing the output of $g$ on $x$ to $a$.
\begin{defn}
	Given an $\fofs$ model $\model=(W,\I,\M,\D, \In)$ and a $w$-variable assignment $a$, the \textit{denotation} of an $n$-tuple of terms $\bar{t}=(t_1,\ldots,t_n)$ is $\deno{\bar{t}}{\model }{g} = (a_1,\cdots,a_n)\in \DW(w)^n$ where \[a_i=\begin{cases} \Inn(t_i)&t_i\in\sigma^\nu\\ g(t_i)&t_i\in V.\end{cases}\]
\end{defn}

We are now equipped to define truth at a world $w$ relative to a $w$-VA. Note that because of the expanding domain assumptions in play, a $w$-VA is automatically a $w'$-VA whenever $w\I w'$ and a $v$-VA whenever $w\M v$. Further, the truth clauses for the connectives and modalities are exactly those of Fischer Servi's original birelational semantics for $\fs$ \cite{fischerservi2}.

\begin{defn}
If $\model=(W,\I,\M,\D, \In)$ is an $\fofs$ model, we define truth of a formula at a world $w$ relative to a $w$-variable assignment function as follows:
\[
\begin{array}{ l l }
	\model,w\Vdash_g P(\bar{t}) &\text{iff} \quad \deno{\bar{t}}{\model}{g} \in \Inp(w)\\
	\model,w\Vdash_g s\equals t &\text{iff} \quad \deno{s}{\model}{g}\DE[w]\deno{t}{\model}{g}\\
	\model,w\Vdash_g (\phi\land \psi) &\text{iff} \quad \model,w\Vdash_g \phi\text{ and }\model,w\Vdash_g \psi\\
	\model,w\Vdash_g (\phi\lor \psi) &\text{iff} \quad \model,w\Vdash_g \phi\text{ or }\model,w\Vdash_g \psi\\
	\model,w\Vdash_g(\phi\to \psi) &\text{iff} \quad \text{for all }w'\succcurlyeq w,\text{ }\model,w'\Vdash_{g} \phi\text{ or }\model,w'\Vdash_{g} \psi\\
	\model,w\Vdash_g \square \phi &\text{iff} \quad \text{for all } 		v'\in(\M\circ\I)(w),\,\model,v'\Vdash_{g}\phi\\
	\model,w\Vdash_g \lozenge \phi &\text{iff} \quad \text{there exists }v\in{\M}(w)\text{ such that }\model,v\Vdash_{g} \phi\\
	\model,w\Vdash_g \forall x \phi &\text{iff} \quad \text{for all }w'\in{\I}(w)\text{, for all }a\in\DW(w'),\,\model,w'\Vdash_{g[x:= a]}\phi\\
	\model,w\Vdash_g \exists x \phi &\text{iff} \quad \text{there exists }a\in\DW(w)\text{ such that }\model,w\Vdash_{g[x:=a]} \phi\\
	\model,w\nVdash_g\bot.&
\end{array}
\]
\end{defn}

\begin{lem}[Persistence]
	\label{persistence}
	If $\model,w\Vdash_g\phi$ for $\phi\in\mathcal{F}(\sigma)$, then for all $w'\in{\I}(w)$ we have $\model,w'\Vdash_g \phi$.
\end{lem}
\begin{proof}
	Induction on formula complexity.
\end{proof}

Unsurprisingly, the choice of assignment is irrelevant for sentences.
\begin{lem}
	Let $\model=(W,\I,\M,\D, \In)$ be an $\fofs$ model and $\phi\in\lang$. If $w\in W$ and $g$ and $g'$ are $w$-VAs, then $\model, w\Vdash_g\phi$ if and only if $\model, w\Vdash_{g'}\phi$.
\end{lem}
\begin{proof}
	Induction on formula complexity.
\end{proof}

We write $\model,w\Vdash\phi$ if $\model,w\Vdash_g\phi$ for each (or equivalently, some) $w$-VA $g$. For a set of sentences $\Gamma$, we write $\model,w\Vdash\Gamma$ if $\model,w\Vdash\phi$ for all $\phi\in \Gamma$. We write $\Gamma\Vdash\phi$ if for all model--world pairs, $\model,w\Vdash \Gamma$ entails $\model,w\Vdash\phi$.
\begin{thm}[Soundness]
	\label{sound}
	If $\Gamma\vdash \phi$, then $\Gamma\Vdash\phi$.
\end{thm}
\begin{proof}
	Each item to be checked is covered by the soundness theorem for $\fs$ \cite{fischerservi} or the soundness theorem for first-order intuitionistic logic presented as a Hilbert system (see, e.g., \cite{dragalin}).
\end{proof}

\section{The Trace Model Construction}
\label{trace}

\subsection{Technical Groundwork}
For this section, fix some signature $\sigma$. For each pair of naturals $(i,j)$ we fix a countably infinite set of constants $C_{i,j}$ such that if $(i,j)\neq (i',j')$, then $C_{i,j}\cap C_{i',j'}=\varnothing$. We set \[\sigma^\nu_{l,m}= \sigma^\nu \cup \left(\bigcup_{i<l\text{ and }j<m} C_{i,j}\right)\] and $\sigma_{l,m}=(\sigma^\nu_{l,m},\sigma^\pi,\sigma^\alpha)$. For convenience, we abbreviate $\lang[\sigma_{l,m}]$ as $\gridlang{l}{m}$. Note that $\gridlang{0}{0}=\lang$. Crucially, we have the following facts. 
\begin{fact}
	\label{intersection}
	$\sigma^\nu_{i+1,j}\cap \sigma^\nu_{i,j+1}=\sigma^\nu_{i,j}$.
\end{fact}

\begin{fact}
	\label{notunion}
	$\sigma^\nu_{i+1,j+1}\setminus(\sigma^\nu_{i,j+1}\cup\sigma^\nu_{i+1,j})$ is infinite.
\end{fact}

Finally, for every pair of naturals $(i,j)$, we choose a distinguished element of $\sigma^\nu_{i,j+1}$, which we call $c^+_{i,j}$. $\gridlang{i}{j}^+=\lang[(\sigma^\nu_{i,j}\cup\{c^+_{i,j}\},\sigma^\pi,\sigma^\alpha)]$. We will say that $\Gamma$ is an $(l,m)$-saturated theory to mean that $\Gamma$ is a saturated theory in the language $\gridlang{l}{m}$. We denote by $\realth{l}{m}$ the set of all $(l,m)$-saturated theories and set $\realthall=\bigcup_{i,j} \realth{i}{j}$. We also define two important binary relations on $\realthall$:
\begin{align*}
	\Gamma\U \Gamma'&\quad\text{iff}\quad \Gamma\subseteq \Gamma'\text{ and for some }l,m\in\nats\text{ we have }\Gamma\in\realth{l}{m}\text{ and }\Gamma'\in\realth{l}{m+1};\\
	\Gamma\R \Delta&\quad\text{iff}\quad\yb{\Gamma}\subseteq \Delta\text{, }\nd{\Gamma}\cap \Delta=\varnothing\text{, and for some }l,m\in\nats\text{ we have }\Gamma\in\realth{l}{m}\text{ and }\Gamma'\in\realth{l+1}{m}.
\end{align*}
Apart from the the restrictions on languages, these relations are defined exactly as the modal and intuitionistic relations are defined for the canonical model of $\fs$. 

The proof of the following lemma is essentially tantamount to verifying that the canonical frame for $\fs$ has (FC2). Though we will require additional argument, this lemma will still be a key ingredient in verifying that the model we construct has (FC2).

\begin{lem}
	\label{fc2}
	Suppose that $\Gamma$, $\Gamma'$, and $\Delta$ are saturated theories such that $\Gamma\R \Delta$ and $\Gamma\U \Gamma'$. Then there exists a saturated theory $\Delta'$ such that $\Gamma'\R\Delta'$ and $\Delta\U \Delta'$.
\end{lem}
\begin{proof}
	Say that $\Gamma$ is an $(l,m)$-saturated theory. First, we wish to show that $(\Delta\cup \yb{\Gamma'},\nd{\Gamma'})$ is a consistent pair. Suppose not. Then $\Delta\vdash \beta\to\delta$ where $\beta\in\yb{\Gamma'}$ and $\delta\in\nd{\Gamma'}$. By \Cref{gengen} we have $\Delta\vdash\forall \bar{x}\qo{(\beta\to\delta)}$ where we have universally quantified out all constants in $\beta\to\delta$ that are not in $\sigma_{l+1,m}$. By \Cref{intersection}, $\forall \bar{x}\qo{(\beta\to\delta)}$ is in $\gridlang{l}{m}$. Therefore, $\Gamma\vdash \lozenge\forall\bar{x}\qo{(\beta\to\delta)}$. By \DFref{barcanish}, $\Gamma\vdash \forall\bar{x}\lozenge\qo{(\beta\to\delta)}$. This implies that $\Gamma'\vdash\forall\bar{x}\lozenge\qo{(\beta\to\delta)}$. By repeated applications of $\Univ$, $\Gamma'\vdash \lozenge(\beta\to\delta)$. But then by $\fstwo$, $\Gamma'\vdash \square\beta\to\lozenge \delta$. Since $\Gamma'$ is closed under deduction, we have $\lozenge\delta\in \Gamma'$, which is a contradiction. We obtain $\Delta'$ by applying \Cref{getreal} to the pair where $\sigma'=\sigma_{l+1,m+1}$.
\end{proof}

\begin{center}
	\begin{figure}[h]
		\begin{tikzpicture}
			\node (gamma) at (0,0) {$\Gamma$};
			\node (Gamma) at (0,1.5) {$\Gamma'$};
			\node (delta) at (1.5,0) {$\Delta$};
			\node (Delta) at (1.5,1.5) {$\Delta'$};
			\draw[-{Square[open, scale=1.4]}] (gamma) -- (delta);
			\draw[-{Arc Barb[reversed, scale=1.4]}] (gamma) -- (Gamma);
			\draw[-{Square[open, scale=1.4]}] (Gamma) -- (Delta);
			\draw[-{Arc Barb[reversed, scale=1.4]}] (delta) -- (Delta);
			\draw[dashed] (.8,.8) rectangle (2,2);
		\end{tikzpicture}
		\caption{A diagram of \Cref{fc2}}
	\end{figure}
\end{center}

We will need to borrow a few more ideas from the propositional case. The next definition and lemma are used in verifying the $\square$ case of the truth lemma for the canonical model for $\fs$. 

\begin{defn}
	Let $\phi\in\gridlang{l}{m}$ and $\Gamma\in\realth{l}{m}$. We say that $\Gamma$ is \textit{$\phi$-averse} if $\square(\phi\lor\psi)\in \Gamma$ implies $\lozenge\psi\in\Gamma$.
\end{defn}
\begin{lem}
	\label{averse}
	Suppose that $\Gamma\in\realth{l}{m}$ is $\phi$-averse. Then there exists $\Delta\in\realth{l+1}{m}$ such that $\Gamma\R \Delta$ and $\phi\notin\Delta$.
\end{lem}
\begin{proof}
	We want to check that the pair $(\yb{\Gamma},\nd{\Gamma}\cup\{\phi\})$ is consistent. If it were not, we would have $\vdash \beta\to(\delta\lor\phi)$ where $\beta\in\yb{\Gamma}$ and $\delta\in\nd{\Delta}$. This implies that $\square(\delta\lor\phi)\in\Gamma$. Because $\Gamma$ is $\phi$-averse, $\lozenge\delta\in \Gamma$, which is a contradiction. We obtain $\Delta$ by applying \Cref{getreal} to the pair where $\sigma'=\sigma_{l+1,m}$.
\end{proof}

\subsection{Modal Chain Extensions}
In the trace model, a single point can have a significant amount of internal structure. Because of this, the extension results that we have proved so far, though useful, are far too weak to actually allow us to construct new points. We begin by defining a stronger sort of extension of a consistent pair.

\begin{defn}
	We say that $\Gamma'$ is a \textit{diamond extension} of a consistent pair $(\Gamma,\Omega)$ if it is an extension of the pair with the property that $\delta\in\nd{\Gamma'}$ implies that $(\Gamma'\cup\{\lozenge\delta\},\Omega)$ is inconsistent.
\end{defn}

Fortunately it is virtually no extra work to construct a diamond extension instead of an ordinary one.

\begin{lem}
	\label{diamondextension}
	If $(\Gamma,\Omega)$ is a consistent pair of formulas in $\gridlang{l}{m+1}\cup\gridlang{l+1}{m}^+$, then it has a diamond extension $\Gamma'\in\realth{l+1}{m+1}$.
\end{lem}
\begin{proof}
	Fix the following enumerations in $\gridlang{l}{m}$ (indices range over $\omega$):
	\begin{align*}
		\delta_i &= \phi^0_i\lor\phi^1_i&\text{for all disjunctive sentences}\\
		\epsilon_i &= \exists x\psi_i(x)&\text{for all existentially quantified sentences}\\
		\zeta_i &=\lozenge\chi_i&\text{for all diamond sentences.}
	\end{align*}
	Let $\Gamma^0=\Gamma$. After constructing $\Gamma^n$, we define $\Gamma^{n+1}$ as follows:
	\begin{itemize}
	\item If $n$ is a multiple of $3$, let $j$ be the least natural such that $\Gamma_n\vdash\delta_j$, but $\Gamma^n\nvdash \phi^0_j$ and $\Gamma^n\nvdash \phi^1_j$. Set \[\Gamma^{n+1} = \begin{cases} 
		\Gamma^n\cup\{\phi_j^0\}&(\Gamma^n\cup\{\phi_j^0\},\Omega)\text{ is consistent}\\
		\Gamma^n\cup\{\phi_j^1\} &\text{else.} \end{cases}\]
	\item If $n$ is one more than a multiple of $3$, let $j$ be the least natural such that $\Gamma^n\vdash\epsilon_j$ but $\Gamma^n\nvdash\psi_j(c)$ for any constant $c$. Choose a constant $c\in C_{l,m}$ such that $c$ does not appear in any of the formulas in $\Gamma^n$ and set $\Gamma^{n+1} =\Gamma_0^n\cup\{\psi(c)\}$.
	\item If $n$ is two more than a multiple of $3$, let $j$ be the least number such that $\Gamma^n\nvdash\zeta_j$ but $(\Gamma^n\cup\{\zeta_j\},\Omega)$ is consistent. We set $\Gamma^{n+1}=\Gamma^n\cup\{\zeta_j\}$.
	\end{itemize}
	Note that $(\Gamma^n,\Omega)$ is consistent for all $n$. The $n=0$ case is by assumption, and if $n$ is a multiple of $3$, $(\Gamma^{n+1},\Omega)$ is consistent by \Cref{disjunctionstep}. If $n$ is one more than a multiple of $3$, $(\Gamma^{n+1},\Omega)$ is consistent by \Cref{existentialstep}. If $n$ is two more than a multiple of $3$, $(\Gamma^{n+1},\Omega)$ is consistent by construction. Setting $\Gamma'$ equal to the deductive closure of $\bigcup_{n<\omega}\Gamma^n$, it should be clear that $\Gamma'$ is a diamond extension of $\Gamma$ disjoint from $\Omega$.
\end{proof}

\begin{defn}
	A \textit{modal chain} is a sequence $\langle \Gamma_0,\ldots, \Gamma_n\rangle$ such that $\Gamma_0\in\realth{0}{m}$ for some $m$ and for each $i<n$, $\Gamma_i\R \Gamma_{i+1}$.
\end{defn}

We will sometimes say \textit{modal $m$-chain} when it is useful to specify that the second coordinate is $m$. We now present an extension result for modal chains. This theorem plays the same roll for us that a Lindenbaum-like extension lemma plays in a canonical model argument. 

\begin{thm}
	\label{getdschain}
	Suppose that $\langle\Gamma_0,\Gamma_1,\cdots,\Gamma_n\rangle$ is a modal $m$-chain and that $(\Gamma_n^*,\Omega_n)$ is a consistent pair in $\gridlang{n}{m}^+$ such that $\Gamma_n\subseteq \Gamma_n^*$. Then there exists a modal $(m+1)$-chain $\langle\Gamma_0',\Gamma_1',\cdots,\Gamma_{n}'\rangle$ such $\Gamma_k\U \Gamma_k'$ for all $0\leq k \leq n$ and $\Gamma_{n}$ is an extension of $(\Gamma_n^*,\Omega_n)$. If moreover $\Omega_n=\{\square\phi\}$ for some $\phi\in \gridlang{n}{m}^+$, $\Gamma_n'$ can be constructed to be $\phi$-averse.
\end{thm}
\begin{proof}
	 Set \[\Omega_{n-1} = \setb{\square\beta\in\gridlang{n-1}{m}}{\Gamma^*_n\vdash\beta\to\biglor_i\gamma_i\text{ for some }\gamma_i\in\Omega_n}\]
	and for $0\leq k< n-1$, set \[\Omega_k=\setb{\square\beta\in\gridlang{k}{m}}{\Gamma_{k+1}\vdash \beta\to\biglor_i\gamma_i\text{ for some }\gamma_i\in\Omega_{k+1}}.\] 
	Our goal is to construct theories $\Gamma'_0\R \Gamma'_1\R\cdots\R\Gamma'_{n}$ such that for each $k$, $\Gamma'_k$ is an extension of $(\Gamma_k,\Omega_k)$ with the following property:
	\begin{center}
		\begin{tabular}{ll}
			 ($*$) & \begin{tabular}{l} If $\delta\in\nd{\Gamma_k'}$ then there exist formulas \\$\beta_i\in\yb{\Omega_k}$ such that $\square(\delta\to\biglor_i \beta_i)\in \Gamma_k'$.\end{tabular}
		\end{tabular}
	\end{center}
	
	First we claim that for each $0\leq k<n$, $(\Gamma_k,\Omega_k)$ is consistent. Suppose that the claim fails for $n-1$. Then $\Gamma_{n-1}\vdash \biglor_i\square \beta_i$ where for each $\beta_i$, $\Gamma_n^*\vdash \beta_i\to\biglor_j \gamma_{i,j}$ for some $\gamma_{i,j}\in\Omega_n$. From \DFref{necdis} we can infer $\Gamma_{n-1}\vdash \square\biglor_i\beta_i$, implying $\Gamma_n\vdash \biglor_i\beta_i$. Since $\Gamma_n$ is prime, for some $i'$, we have $\Gamma_n\vdash \beta_{i'}$. But then $\Gamma^*_n\vdash \beta_{i'}$, so $\Gamma^*_n\vdash\biglor_j\gamma_{i',j}$, which contradicts the consistency of $(\Gamma_n^*,\Omega_n)$. 
	
	Now assume that $(\Gamma_{k+1},\Omega_{k+1})$ is consistent. If $(\Gamma_k,\Omega_k)$ were inconsistent, we would have $\Gamma_k\vdash\biglor_i\square\beta_i$, where $\square \beta_i\in\Omega_k$. This implies that $\Gamma_k\vdash \square\biglor_i\beta_i$, so $\Gamma_{k+1}\vdash \biglor_i\beta_i$. By primeness, we find an $i'$ such that $\Gamma_{k+1}\vdash \beta_{i'}$. From the definition of $\Omega_k$, we see that $\Gamma_{k+1}\vdash \beta_{i'}\to \biglor_j \gamma_{j}$ where $\gamma_j\in\Omega_{k+1}$. But then $\Gamma_{k+1}\vdash\biglor_j \gamma_{j}$, which contradicts the consistency of $(\Gamma_{k+1},\Omega_{k+1})$.
	
	Let $\Gamma'_0$ be a diamond-extension of $(\Gamma_0,\Omega_0)$. To check ($*$), suppose that $\delta\in\nd{\Gamma'_0}$. Then we have $\Gamma_0'\vdash \lozenge\delta\to\biglor_i\square \beta_i$ for some $\square\beta_i\in \Omega_0$. This implies $\Gamma_0'\vdash \lozenge\delta\to\square \biglor_i \beta_i$ By $\fsone$, $\Gamma'_0\vdash \square (\delta\to \biglor_i\beta_i)$.
	
	We construct the rest of the $\Gamma'_k$s recursively. Say that $\Gamma'_k$ has already been constructed, and suppose that the pair $(\Gamma_{k+1}\cup\yb{\Gamma_k'},\Omega_{k+1}\cup\nd{\Gamma_k'})$ is inconsistent. Then
	\begin{align}
		\Gamma_{k+1}\cup\yb{\Gamma_k'}\vdash (\biglor_i\gamma''_i)\lor\delta
	\end{align}
	where each $\gamma''_i\in\Omega_{k+1}$ and $\delta\in \nd{\Gamma'_k}$. By ($*$), $\Gamma'_k\vdash\square( \delta\to\biglor_i\beta'_i)$ for some $\beta_i'\in\yb{\Omega_k}$. Each $\beta_i'$ has the property that $\Gamma_{k+1}\vdash \beta_i'\to\biglor_j \gamma'_{i,j}$ where $\gamma'_{i,j}\in\Omega_{k+1}$. Therefore
	\begin{align}
		\Gamma_{k+1}\cup\yb{\Gamma'_k}\vdash \delta\to\biglor_i\biglor_j\gamma'_{i,j}.
	\end{align}
	Combining (1) and (2), we have	$\Gamma_{k+1}\cup\yb{\Gamma'_k}\vdash (\biglor_i\gamma''_i)\lor(\biglor_i\biglor_j\gamma'_{i,j})$. As each $\gamma''_i$ and $\gamma'_{i,j}$ is in $\Omega_{k+1}$, for cleanliness of notation, we reindex to write $\Gamma_{k+1}\cup\yb{\Gamma'_k}\vdash \biglor_i\gamma_i$ where each $\gamma_i\in \Omega_{k+1}$. 
	Therefore, $\Gamma_{k+1}\vdash\beta\to \biglor_i\gamma_i$ for some $\beta\in\yb{\Gamma'_k}$. We can quantify out all of the constants $\bar{c}=\langle c_1,\ldots,c_p\rangle$ in $\beta$ that are not in $\sigma_{k,m}$ to obtain $\Gamma_{k+1}\vdash \exists\bar{x}\tilde{\beta}\to(\biglor_i\gamma_i)\lor(\biglor_i\gamma'_i)$. This implies that $\square\exists\bar{x}\qo{\beta}\in \Omega_k$. But then $\square\exists\bar{x}\qo{\beta}\notin\Gamma'_k$ while $\square\beta\in\Gamma'_k$. This contradicts the consistency of $\Gamma'_k$ since $\vdash \square\beta\to\square\exists\bar{x}\qo{\beta}$ by regularity.
	We take $\Gamma_{k+1}'$ to be some diamond-extension of $(\Gamma_{k+1},\Omega_{k+1})$. Now we check that $\Gamma_{k+1}$ has ($*$). Suppose $\delta\in\nd{\Gamma'_{k+1}}$. Then $\Gamma'_{k+1}\vdash \lozenge\delta\to (\biglor\gamma''_i)\lor\delta'$ where each $\gamma''_i\in\Omega_{k+1}$ and $\delta'\in\nd{\Gamma_k}$. By ($*$), $\Gamma_k'\vdash \square(\delta'\to\biglor_i \beta'_i)$ where $\beta'_i\in\yb{\Omega_k}$, so $\Gamma'_{k+1}\vdash\delta'\to\biglor\beta_i'$. For each $\beta_i'$, there are $\gamma'_{i,j}\in \Omega_{k+1}$ such that $\Gamma_{k+1}\vdash\beta'_i\to\biglor_j \gamma'_{i,j}$. Therefore, $\Gamma'_{k+1}\vdash \lozenge \delta \to ((\biglor_i \gamma''_i)\lor (\biglor_i\biglor_j \gamma'_{i,j}))$. Again, for convenience, we reindex to obtain $\Gamma'_{k+1}\vdash \lozenge \delta\to\biglor_i\gamma_i$ where each $\gamma_i\in\Omega_{k+1}$. Each $\gamma_i$ can be written as $\square\beta_i$ by the definition of $\Omega_{k+1}$, so we have $\Gamma'_{k+1}\vdash \lozenge \delta\to\biglor_i\square\beta_i$, in turn implying $\Gamma'_{k+1}\vdash \lozenge \delta\to\square\biglor_i\beta_i$. By $\fsone$, $\Gamma'_{k+1}\vdash \square(\delta\to \biglor_i\beta_i)$, as desired.
	
	To construct $\Gamma_n'$, we merely need to check that $(\Gamma^*_n\cup\yb{\Gamma'_{n-1}},\Omega_n\cup\nd{\Gamma'_{n-1}})$ is consistent. The argument will be largely the same as in the construction of the other $\Gamma_k'$s. If it were not consistent, we would have $\Gamma_n^*\cup\yb{\Gamma'_{n-1}}\vdash (\biglor_i \gamma_i'')\lor\delta$ for some $\gamma_i''\in\Omega_n$ and $\delta\in\nd{\Gamma'_{n-1}}$. By ($*$), $\Gamma_{n-1}'\vdash \square(\delta\to\biglor\beta_i')$ for some $\beta_i'\in\yb{\Omega_{n-1}}$, so $\Gamma_{n}^*\cup\yb{\Gamma_{n-1}'}\vdash \biglor_i \gamma_i$ for some $\gamma_i\in\Omega_n$. Then $\Gamma_{n}^*\vdash \beta\to\biglor_i\gamma_i$ for some $\beta\in \yb{\Gamma_{n-1}'}$. Again by \Cref{exsuff}, $\Gamma_{n}^*\vdash \exists \bar{x}\qo{\beta} \to \biglor_i\gamma_i$ where $\bar{c}$ are the constants in $\beta$ not in $\sigma^\nu_{n,m}$ and $\bar{x}$ are fresh variables. This implies that $\square\exists \bar{x}\qo{\beta}\in \Omega_{n-1}$, contradicting that $\square\beta\in \Gamma_{n-1}'$. By \Cref{getreal}, we obtain $\Gamma_n'$.
	
	Finally suppose that $\Omega_n=\{\square\phi\}$ for some $\phi\in \gridlang{n}{m}$. Since we have already checked that $(\Gamma^*_n\cup\yb{\Gamma'_{n-1}},\Omega_n\cup\nd{\Gamma'_{n-1}})$ is consistent, we can apply \Cref{diamondextension} to obtain a diamond extension $\Gamma'_n$. Now we check that $\Gamma_n'$ is $\phi$-averse. Suppose that $\square(\phi\lor\psi)\in\Gamma_n'$. If $\lozenge\psi\notin\Gamma'_n$, then we must have that $\Gamma_n'\vdash \lozenge \psi\to (\square\phi\lor\delta)$ for $\delta\in\nd{\Gamma'_{n-1}}$. By ($*$), $\Gamma'_{n-1}\vdash \square(\delta\to\biglor_i\gamma_i)$ where $\gamma_i\in\yb{\Omega_{n-1}}$. Each $\gamma_i$ then has the property that $\Gamma^*_n\vdash \gamma_i\to\square\phi$. Combining these facts, $\Gamma_n'\vdash \lozenge\psi\to\square\phi$. By $\fsone$, $\Gamma_n'\vdash \square(\psi\to\phi)$, so reasoning under the box affords us $\Gamma_n'\vdash\square\phi$. This a contradiction, so $\lozenge\psi\in\Gamma'_n$, as desired.
\end{proof}

\begin{center}
	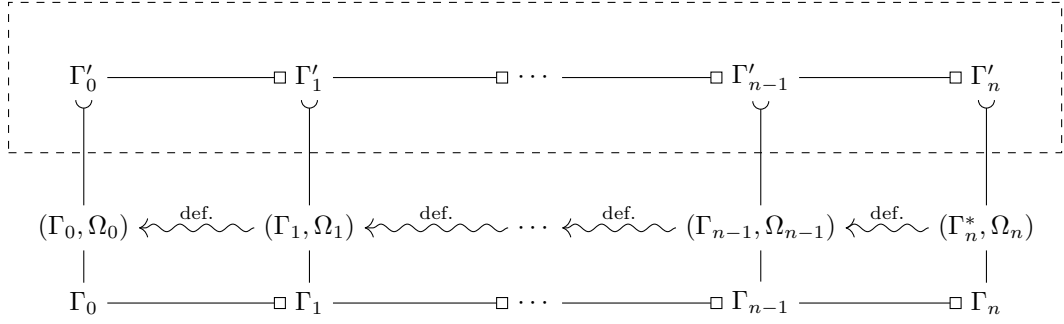
\begin{figure}
		\begin{tikzpicture}
			\node (gamma0) at (0,0) {$\Gamma_0$};
			\node (gamma1) at (3,0) {$\Gamma_1$};
			\node (lell) at (6,0) {$\cdots$};
			\node (gamman-1) at (9,0) {$\Gamma_{n-1}$};
			\node (gamman) at (12,0) {$\Gamma_n$};
			\node (gamma0p) at (0,1) {$(\Gamma_0,\Omega_0)$};
			\node (gamma1p) at (3,1) {$(\Gamma_1,\Omega_1)$};
			\node (mell) at (6,1) {$\cdots$};
			\node (gamman-1p) at (9,1) {$(\Gamma_{n-1},\Omega_{n-1})$};
			\node (gammanp) at (12,1) {$(\Gamma^*_n,\Omega_n)$};
			\node (gamma0') at (0,3) {$\Gamma_0'$};
			\node (gamma1') at (3,3) {$\Gamma_1'$};
			\node (hell) at (6,3) {$\cdots$};
			\node (gamman-1') at (9,3) {$\Gamma_{n-1}'$};
			\node (gamman') at (12,3) {$\Gamma_n'$};
			\draw[-{Square[open, scale=1.4]}] (gamma0) -- (gamma1);
			\draw[-{Square[open, scale=1.4]}] (gamma1) -- (lell);
			\draw[-{Square[open, scale=1.4]}] (lell) -- (gamman-1);
			\draw[-{Square[open, scale=1.4]}] (gamman-1) -- (gamman);
			\draw (gamma0) -- (gamma0p);
			\draw (gamma1) -- (gamma1p);
			\draw (gamman-1) -- (gamman-1p);
			\draw (gamman) -- (gammanp);
			\draw[-{Square[open, scale=1.4]}] (gamma0') -- (gamma1');
			\draw[-{Square[open, scale=1.4]}] (gamma1') -- (hell);
			\draw[-{Square[open, scale=1.4]}] (hell) -- (gamman-1');
			\draw[-{Square[open, scale=1.4]}] (gamman-1') -- (gamman');
			\draw[-{Arc Barb[reversed, scale=1.4]}] (gamma0p) -- (gamma0');
			\draw[-{Arc Barb[reversed, scale=1.4]}] (gamma1p) -- (gamma1');
			\draw[-{Arc Barb[reversed, scale=1.4]}] (gamman-1p) -- (gamman-1');
			\draw[-{Arc Barb[reversed, scale=1.4]}] (gammanp) -- (gamman');
			\draw[-{To[scale=1.4]}, decorate, decoration={snake, amplitude=.5mm}] (gamma1p) -- node [above,midway] {\scriptsize def.} (gamma0p);
			\draw[-{To[scale=1.4]}, decorate, decoration={snake, amplitude=.5mm}] (mell) -- node [above,midway] {\scriptsize def.} (gamma1p);
			\draw[-{To[scale=1.4]}, decorate, decoration={snake, amplitude=.5mm}] (gamman-1p) -- node [above,midway] {\scriptsize def.} (mell);
			\draw[-{To[scale=1.4]}, decorate, decoration={snake, amplitude=.5mm}] (gammanp) -- node [above,midway] {\scriptsize def.} (gamman-1p);
			\draw[dashed] (-1,2) rectangle (13,4);
		\end{tikzpicture}
	\caption{A schematic of the construction in the proof of \Cref{getdschain}}
	\end{figure}
\end{center}

\subsection{Records}
As was previous alluded to, points in the trace model are fairly structured. In particular, a point in the trace model will be a certain kind of function into $\realthall$. We begin by pinning down which sets will serve as an appropriate domain for such a function.

\begin{defn}
	\textit{An $(l,m)$-Ferrers set} $\Lambda$ is a subset of $[0,l]\times[0,m]$ such that the following hold:
	\begin{enumerate}[(i)]
		\item $(0,0)\in \Lambda$.
		\item $(l,m)\in \Lambda$.
		\item For all $i>0$ and all $j$, if $(i,j)\in \Lambda$ then $(i-1,j)\in \Lambda$.
		\item For all $i$ and all $j< m$, if $(i,j)\in \Lambda$ then $(i,j+1)\in \Lambda$.
	\end{enumerate}
\end{defn}

As an aside, Ferrers sets are named for their similarity to Ferrers diagrams, which are combinatorial objects used in the study of integer partitions (see, e.g., \cite[p.~220--221]{harrisetal}).

\begin{figure}[h]
	\begin{center}
		\begin{tikzpicture}[modal]
			\filldraw[color=black,thick,fill=white,fill opacity=.5,rounded corners] (-.25,-.25) -- (-.25,.75) -- (.75,.75) -- (.75,.25) -- (.25,.25) -- (.25,-.25) -- cycle;
			\filldraw[color=black,thick,fill=white,fill opacity=.5,rounded corners] (2.75,-.75) -- (2.75, 1.25) -- (4.25, 1.25) -- (4.25, .75) -- (3.75, .75) -- (3.75, -.25) -- (3.25, -.25) -- (3.25, -.75) -- cycle;
			\filldraw[color=black,thick,fill=white,fill opacity=.5,rounded corners] (6.25,-.5) -- (6.25, 1) -- (7.75,1) -- (7.75, -.5) -- cycle;
			\filldraw[color=black,thick,fill=white,fill opacity=.5,rounded corners] (9.75,-.25) -- (9.75, .75) -- (11.75,.75) -- (11.75, .25) -- (10.75, .25) -- (10.75, -.25) -- cycle;
			\node[point] (00) at (0,0) {};
			\node[point] (01) at (0,.5) {};
			\node[point] (11) at (0.5,0.5) {};
			\node[point] (00*) at (3,-.5) {};
			\node[point] (01*) at (3,0) {};
			\node[point] (02*) at (3,.5) {};
			\node[point] (03*) at (3,1) {};
			%\node[point] (10*) at (3.5,-.5) {};
			\node[point] (11*) at (3.5,0) {};
			\node[point] (12*) at (3.5,.5) {};
			\node[point] (13*) at (3.5,1) {};
			%\node[point] (21*) at (4,0) {};
			%\node[point] (22*) at (4,.5) {};
			\node[point] (23*) at (4,1) {};
			\node[point] (00**) at (6.5,-.25) {};
			\node[point] (01**) at (6.5,.25) {};
			\node[point] (02**) at (6.5,.75) {};
			\node[point] (10**) at (7,-.25) {};
			\node[point] (11**) at (7,.25) {};
			\node[point] (12**) at (7,.75) {};
			\node[point] (20**) at (7.5,-.25) {};
			\node[point] (21**) at (7.5,.25) {};
			\node[point] (22**) at (7.5,.75) {};
			\node[point] (00***) at (10,0) {};
			\node[point] (01***) at (10,.5) {};
			\node[point] (10***) at (10.5,0) {};
			\node[point] (11***) at (10.5,.5) {};
			\node[point] (21***) at (11,.5) {};
			\node[point] (31***) at (11.5,.5) {};
		\end{tikzpicture}
	\end{center}
	\caption{Examples of Ferrers sets where $(0,0)$ is the bottom left point, the first coordinate tracks horizontally, and the second vertically}
\end{figure}
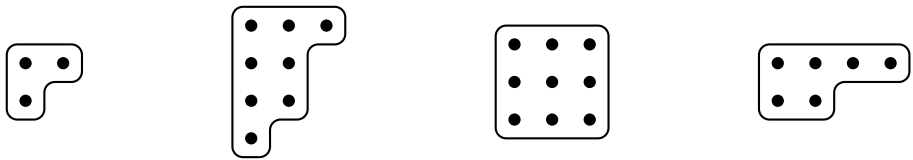

\begin{defn}
	For any naturals $l$ and $m$, an \textit{$(l,m)$-record} $\rho$ is a function $\Lambda\to \realthall$ such that the following hold:
	\begin{enumerate}[(i)]
		\item $\Lambda$ is an $(l,m)$-Ferrers set.
		\item For all $(i,j)\in \Lambda$, $\rho(i,j)\in\realth{i}{j}$.
		\item For all $(i,j)\in \Lambda$, if $(i+1,j)\in F$, then $\rho(i,j)\R \rho(i+1,j)$.\label{recdefmod}
		\item For all $(i,j)\in \Lambda$ with $j<m$, we have $\rho(i,j)\U \rho(i,j+1)$.
	\end{enumerate}
\end{defn}  

Equivalently, part (\ref{recdefmod}) of the definition could be replaced by requiring that for each $j\leq m$, $\rho(-,j)$ is a modal chain. In fact, the proof of the truth lemma will frequently apply the main extension theorem for modal chains in the previous subsection to the top row of a record. We will use the standard pieces of notation $\dom\rho$ for the domain of a record $\rho$ and $\rho|_\Lambda$ for the record $\rho$ restricted to the Ferrers set $\Lambda\subseteq \dom\rho$.

\begin{defn}
	We say that the $(l',m')$-record \textit{end extends} an $(l,m)$-record $\rho$ if $\dom \rho = \dom \rho' \cap([0,l]\times [0,m])$ and $\rho = \rho'|_{\dom \rho}$.
\end{defn}

\begin{lem}
	Suppose that $\rho'$ is an $(l',m')$-record. If $(l,m)\in \dom \rho'$, then the restriction of $\rho'$ to $\dom \rho' \cap([0,l]\times [0,m])$ is an $(l,m)$-record.
\end{lem}
\begin{proof}
	Immediate.
\end{proof}

If $\rho$ is an $(l,m)$-record, we will write $\ur{\rho}$ for $\rho(l,m)$, which is just the saturated theory associated to the upper-right corner of the underlying Ferrers set.

\subsection{The Trace Model}
We are primed to define the trace model. As a notational convention, we will use fraktur lettering to indicate that we are looking at the trace model, trace frame, or some component thereof. 

\label{tracemod}
\begin{defn}
	\label{tdef}
	The \textit{trace frame} of $\fofs$ is $\trf=(\WC, \IC, \MC, \DC)$, where:
	\begin{enumerate}[(i)]
		\item $\WC$ is the set of all records;
		\item Where $\rho$ is an $(l,m)$-record and $\rho'$ is an $(l',m')$-record, $\rho\IC \rho'$ just in case $l=l'$ and $\rho'$ end extends $\rho$;
		\item \label{tmodaldef} Where $\rho$ is an $(l,m)$-record and $\rho'$ is an $(l',m')$-record, $\rho\MC \rho'$ just in case $l'=l+1$, $m=m'$, and $\rho'$ end extends $\rho$;
		\item $\DC=(\DWC,\DEC)$ is given by
		\begin{itemize}
			\item $\DWC(\rho)=\sigma_{l,m}^\nu$ for any $(l,m)$-record $\rho$;
			\item $c_1 \DEC[\rho] c_2$ just in case $c_1\equals c_2\in\ur{\rho}$.
		\end{itemize}
	\end{enumerate}
\end{defn}

Intuitively, the difficulty of constructing a Fischer-Servi frame in the first-order case comes from the fact that one has to build worlds in the backwards direction. The condition (FC1) demands that a certain world has a modal predecessor, but since we have expanding signatures (mirroring the expanded domains), all of the constants in the language of a predecessor world already belong to whichever world we started with. Consequently, there are no constants available with which to carry out the Henkinization part of our saturation procedure. To avoid this, we shift from saturated theories to records, which encode their own modal predecessors from the onset. The following proposition shows that this strategy does indeed circumvent the complication.

\begin{prop}
	\label{checkframe}
	The trace frame is an $\fofs$ frame.
\end{prop}

\begin{proof}
	First we verify that $(\WC,\IC,\MC)$ is a Fischer Servi frame. It is clear that $\IC$ is a partial order. To check (FC1), suppose that $\rho\MC \tau$ and $\tau\IC \tau'$. Say that $\rho$ is an $(l,m)$-record and that $\tau'$ is an $(l+1,m')$-record. Setting $\rho'=\tau'|_{l,m'}$, we clearly have $\rho\IC \rho'$ and $\rho'\IC \tau'$. For (FC2), suppose that $\rho\IC \rho'$ and $\rho\MC \tau$ where $\rho\in\realth{l}{m}$. We prove that there exists $\tau'$ with $\tau\IC\tau'$ and $\rho'\MC\tau'$ by induction on $m'-m$, where $\rho'\in \realth{l}{m'}$. If $m'-m=1$, then we have $\ur{\rho}\U\ur{\rho'}$ and $\ur{\rho}\R\ur{\tau}$. We apply \Cref{fc2} to obtain an $(l+1,m')$-saturated theory $\Gamma$ such that $\ur{\tau}\U\Gamma$ and $\ur{\rho'}\R\Gamma$. Then, we define a record $\tau':\dom \tau\cup ([0,l+1]\times\{m'\})\to\realthall$ by
	\[\tau'(i,j)=\begin{cases} \tau(i,j)&j\leq m \\ 
		\rho'(i,j)&i\leq l\text{ and }j=m'\\
		\Gamma &\text{else.}\end{cases}\]
	Now suppose that $m'-m>1$. Consider the record $\rho''=\rho_{l,m'-1}$. We have $\rho\IC\rho''$ and $\rho\MC\tau$, so by induction we obtain an $(l+1,m'-1)$-record $\tau''$ such that $\tau\IC\tau''$ ad $\rho''\MC \tau ''$. Since we also have $\rho''\IC\rho'$, we are essentially back in the base case, so we apply an identical argument to obtain an $(l+1,m')$-record $\tau'$ such that $\tau''\IC\tau'$ and $\rho'\MC\tau'$. By the transitivity of $\IC$, $\tau\IC\tau'$, so we are done. 
\end{proof}

In the proof of the truth lemma, we will at times need to shift to an expanded language, which then also forces us to only look at records with sufficiently large indices. Accordingly, we will have to make use of models built on restrictions of the trace frame.

\begin{defn}
	The \textit{$(l,m)$-partial trace frame} $\trfp{l}{m}$ of $\fofs$ is the subframe of $\mathcal{T}$ where the underlying set is restricted to the set of $(l',m')$-records where $l'\geq l$ and $m'\geq m$.
\end{defn}

Note that the $(0,0)$-partial trace frame is just the trace frame. Luckily, it is quite easy to see that these restricted frames are still in the correct class.

\begin{prop}
	For any $l,m\in\nats$, $\trfp{l}{m}$ is an $\fofs$ frame.
\end{prop}
\begin{proof}
	The set of records in question is closed under $\IC$ and $\MC$, so by \Cref{frres} and \Cref{checkframe}, $\trfp{l}{m}$ is an $\fofs$ frame.
\end{proof}

\begin{center}
	\begin{figure}
		\begin{tikzpicture}[modal]
			\node[point,fill=red!50!yellow] (w) at (-3.25,0.25) {};
			\node[point,fill=red] (w') [above=1cm of w] {};
			\node[point,fill=yellow] (x) [right=1cm of w] {};
			\path[->] (w) edge[color=blue] (x);
			\path[->] (w) edge (w');
			\fill[fill=red!50!white,rounded corners] (-.25,-.75) -- (-.25,2.25) -- (1.25,2.25) -- (1.25,.75) -- (.75,.75) -- (.75,-.25) -- (.25,-.25) -- (.25,-.75) -- cycle;
			\fill[fill=yellow!50!white,rounded corners] (-.25,-.75) -- (-.25,1.25) -- (1.75,1.25) -- (1.75,.75) -- (.75,.75) -- (.75,-.25) -- (.25,-.25) -- (.25,-.75) -- cycle;
			\fill[fill=orange!50!white,rounded corners] (-.25,-.75) -- (-.25,1.25) -- (1.25,1.25) -- (1.25,.75) -- (.75,.75) -- (.75,-.25) -- (.25,-.25) -- (.25,-.75) -- cycle;
			\fill[fill=orange!50!white] (1.25,1.25) -- (1.25,1) -- (1,1) -- (1,1.25) -- cycle;
			\draw[color=black,thick,rounded corners] (-.25,-.75) -- (-.25,2.25) -- (1.25,2.25) -- (1.25,.75) -- (.75,.75) -- (.75,-.25) -- (.25,-.25) -- (.25,-.75) -- cycle;
			\draw[color=black,thick,rounded corners] (-.25,-.75) -- (-.25,1.25) -- (1.75,1.25) -- (1.75,.75) -- (.75,.75) -- (.75,-.25) -- (.25,-.25) -- (.25,-.75) -- cycle;
			\node[point] (00) at (0,-.5) {};
			\node[point] (01) at (0,0) {};
			\node[point] (02) at (0,.5) {};
			\node[point] (03) at (0,1) {};
			\node[point] (04) at (0,1.5) {};
			\node[point] (05) at (0,2) {};
			\node[point] (11) at (0.5,0) {};
			\node[point] (12) at (0.5,.5) {};
			\node[point] (13) at (0.5,1) {};
			\node[point] (14) at (0.5,1.5) {};
			\node[point] (15) at (0.5,2) {};
			\node[point] (23) at (1,1) {};
			\node[point] (24) at (1,1.5) {};
			\node[point] (25) at (1,2) {};
			\node[point] (33) at (1.5,1) {};
			\node[xpoint] (34) at (1.5,1.5) {};
			\node[xpoint] (35) at (1.5,2) {};
			\filldraw[color=black,thick,fill=green!50!white,rounded corners] (3.75,-.75) -- (3.75,2.25) -- (5.75,2.25) -- (5.75,.75) -- (4.75,.75) -- (4.75,-.25) -- (4.25,-.25) -- (4.25,-.75) -- cycle;
			\node[point] (00*) at (4,-.5) {};
			\node[point] (01*) at (4,0) {};
			\node[point] (02*) at (4,.5) {};
			\node[point] (03*) at (4,1) {};
			\node[point] (04*) at (4,1.5) {};
			\node[point] (05*) at (4,2) {};
			\node[point] (11*) at (4.5,0) {};
			\node[point] (12*) at (4.5,.5) {};
			\node[point] (13*) at (4.5,1) {};
			\node[point] (14*) at (4.5,1.5) {};
			\node[point] (15*) at (4.5,2) {};
			\node[point] (23*) at (5,1) {};
			\node[point] (24*) at (5,1.5) {};
			\node[point] (25*) at (5,2) {};
			\node[point] (33*) at (5.5,1) {};
			\node[point] (34*) at (5.5,1.5) {};
			\node[point] (35*) at (5.5,2) {};
			\node[point,fill=red!50!yellow] (W) at (7.75,0.25) {};
			\node[point,fill=red] (W') [above=1cm of W] {};
			\node[point,fill=yellow] (X) [right=1cm of W] {};
			\node[point,fill=green] (X') [right=1cm of W'] {};
			\path[->] (W) edge[color=blue] (X);
			\path[->] (W) edge (W');
			\path[->] (W') edge[color=blue] (X');
			\path[->] (X) edge (X');
			\node (a1) at (-1.25,.75) {$\Longrightarrow$};
			\node (a2) at (2.75,.75) {$\Longrightarrow$};
			\node (a3) at (6.75,.75) {$\Longrightarrow$};
		\end{tikzpicture}
	\caption{A diagram representing some particular instance of the argument in the proof of \Cref{checkframe} that (FC2) holds for the trace frame}
	\end{figure}
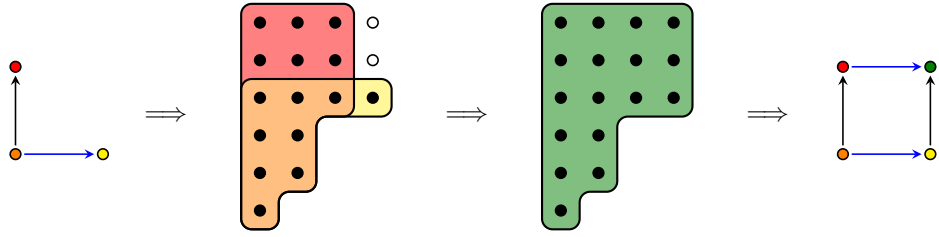
\end{center}
\begin{defn}
	The \textit{trace model} of $\fofs$ is $\trm=(\WC, \IC, \MC, \DC, \InC)$ where $\InC=(\InnC,\InpC)$ is defined as follows:
	\begin{enumerate}[(i)]
		\item Where $\rho$ is an $(l,m)$-record, $\InnC(\rho):\sigma^\nu\to\sigma^\nu_{l,m}$ is given by $\InnC(\rho)(c)=c$.
		\item Where $\rho$ is an $(l,m)$-record and $P$ has arity $n$, $\InpC(\rho):\sigma^\pi_{l,m}\to\pso{\DWC(\rho)}$ is the function given by $P\mapsto\setb{\bar{c}\in \DWC(\rho)^n}{P(\bar{c})\in\ur{\rho}}$.
	\end{enumerate}
\end{defn}

\begin{prop}
	$\trm$ is an $\fofs$ model.
\end{prop}

\begin{proof}
	Since \Cref{checkframe} establishes that $\trf$ is an $\fofs$ frame, what remains is to check that $\InC$ is an interpretation on $\trf$. For $c\in\sigma^\nu$, $\InnC(c)=c$ is in the domain assigned to an arbitrary $(l,m)$-record by the definitions of $\gridlang{l}{m}$ and $\DC$. Now we check that $\InpC$ satisfies the properties listed in \Cref{inter}.
	\begin{itemize}
		\item If $\sigma^\alpha(P)=n$, then it is clear from the definition of $\InpC$ that $\InpC(\rho)(P)$ will be a subset of $\ps{\DWC(\rho)^n}$
		\item Suppose that $\rho\IC\rho'$ and that $\bar{c}\in \InpC(w)(P)$. We want to show that $\bar{c}\in \InpC(w')(P)$. From the definition of $\InpC(w)$, we have that $P(\bar{c})\in \ur{\rho}$. Since $\ur{\rho}\subseteq \ur{\rho'}$, we have $P(\bar{c})\in\ur{\rho'}$, so $\bar{c}\in\InpC(w')(P)$.
		\item Suppose that $c_i \DEC[\rho] c'_i$ for $1\leq i \leq n$ and that $\bar{c}\in \InpC(\rho)(P)$. We want to show that $\bar{c'}\in\InpC(\rho)(P)$, as well. We must have $P(\bar{c})\in \ur{\rho}$. Since also have $c_i\equals c'_i\in\ur{\rho}$ for each $i$, the axiom $\idsub$ and deductive closure of $\ur{\rho}$ guarantees that $P(\bar{c'})\in\ur{\rho}$. Therefore, $\bar{c'}\in\InpC(\rho)(P)$.\qedhere
	\end{itemize}
\end{proof}

\begin{defn}
	The \textit{$(l,m)$-partial trace model} $\trmp{l}{m}$ of $\fofs$ is the submodel $\trm$ obtained by restricting to $\trfp{l}{m}$ and for an $(l',m')$-record $\rho$ defining $\InnC_{l,m}(\rho)$ to be the function $\InnC_{l,m}(\rho):\sigma_{l,m}^\nu\to\sigma^\nu_{l',m'}$ given by $\InnC(\rho)(c)=c$.
\end{defn}

\begin{prop}
	For any $l,m\in\nats$, $\trmp{l}{m}$ is an $\fofs$ model.
\end{prop}

\begin{proof}
	An argument virtually identical to the proof of \Cref{checkframe} suffices.
\end{proof}

In order to make use of these partial models, we must quickly check that passing to them respects truth.

\begin{prop}
	\label{partialtruth}
	If $\phi\in\gridlang{l}{m}$ and $\rho$ is an $(l',m')$-record for $l'\geq l$ and $m'\geq m$, then $\trmp{l}{m},\rho\Vdash\phi$ if and only if $\trmp{l'}{m'},\rho\Vdash\phi$.
\end{prop}
\begin{proof}
	Consider the model $\model=\trmp{l}{m}(\setb{\rho}{\rho\text{ is an } (l',m')\text{-record}})$. We can see that $\model$ is a generated submodel of $\trmp{l}{m}$ and a reduct of $\trmp{l'}{m'}$, so we have $\trmp{l}{m},\rho\Vdash\phi$ if and only if $\model,\rho\Vdash\phi$ if and only if $\trmp{l'}{m'},\rho\Vdash\phi$.
\end{proof}

\begin{lem}[Truth]
	For all $l,m\in\nats$ if $\rho$ is an $(l,m)$-record $\phi\in\gridlang{l}{m}$, then $\trmp{l}{m},\rho\Vdash\phi$ if and only if $\phi\in\ur{\rho}$.
\end{lem}
\begin{proof}	
	We proceed by induction on formula complexity.
	
	The cases for atomic sentences, conjunction, disjunction, and sentences $\exists x\phi$ where $x$ is not free in $\phi$ are immediate. The case for a sentence $\forall x\phi$ where $x$ is not free in $\phi$ follows from \Cref{persistence}. 
	
	$\phi=\psi\to\chi$. Suppose that $\psi\to\chi\in \ur{\rho}$. For any record $\tau$ with $\rho\IC \tau$ such that $\trmp{l}{m}, \tau\Vdash\psi$, we have $\psi\in\ur{\tau}$ by induction. Since $\psi\to\chi\in\ur{\tau}$ by the definition of $\IC$, we have $\chi\in \ur{\tau}$, which in turn implies $\trmp{l}{m}, \tau\Vdash\chi$. Therefore, $\trmp{l}{m},\rho\Vdash \psi\to\chi$. In the other direction, suppose that $\psi\to\chi\notin\ur{\rho}$. This implies that $(\ur{\rho}\cup\{\psi\},\{\chi\})$ is a consistent pair. We apply \Cref{getdschain} to the top level of $\rho$ and the pair in question, which allows us to extend $\rho$ to an $(l,m+1)$-record $\rho'$ such that $\psi\in\ur{\rho'}$ and $\chi\notin\ur{\rho'}$. By the induction hypothesis, $\trmp{l}{m},\rho'\Vdash \psi$ but $\trmp{l}{m},\rho'\nVdash\chi$, so $\trmp{l}{m},\rho\nVdash\psi\to\chi$.
	
	$\phi=\lozenge\psi$. If $\lozenge\psi\in \ur{\rho}$, we want to argue that $(\yb{\ur{\rho}}\cup\{\psi\},\nd{\ur{\rho}})$ is consistent. If it were not, $\ur{\rho'}\vdash \square(\psi\to \delta)$ where $\delta\in \nd{\ur{\rho}}$. By \DFref{fstworeverse}, $\ur{\rho'}\vdash\lozenge\delta$, which is a contradiction. Therefore, we can apply \Cref{getreal} to obtain $\Gamma\in \realth{l+1}{m}$ extending this pair. Since $\ur{\rho}\R\Gamma$ we can extend $\rho$ by $\Gamma$ in the following way to obtain the $(l+1,m)$-record $\tau:\dom\rho\cup\{(l+1,m)\}\to\realthall$:
	\[\tau(i,j) = \begin{cases} \Gamma & (i,j)=(l+1,m) \\ \rho(i,j) & \text{else}.\end{cases}\]
	Note that $\psi\in\ur{\tau}$. By the induction hypothesis, $\trmp{l}{m},\tau\Vdash\psi$. Since $\rho\MC \tau$, $\trmp{l}{m},\rho\Vdash\lozenge\psi$. Now assume $\lozenge\psi\notin\ur{\rho}$. By the definition of $\MC$, for every $\tau$ with $\rho\MC \tau$, we must have $\psi\notin \ur{\tau}$. By induction, for every such $\tau$, $\trmp{l}{m},\tau\nVdash\psi$, so $\trmp{l}{m},\rho\nVdash\lozenge\psi$.
	
	$\phi=\square\psi$. Suppose that $\square\phi\in \ur{\rho}$ and that $\rho\IC \rho'\MC \tau'$. We will also have $\square\phi\in\ur{\rho'}$, so by the definition of $\MC$, $\phi\in \ur{\tau'}$. By induction $\trmp{l}{m},\tau'\Vdash \phi$, so $\trmp{l}{m},\rho\Vdash\square\phi$. If $\square\phi\notin\ur{\rho}$, we apply \Cref{getdschain} to the top row of $\rho$ and $(\ur{\rho},\{\square\phi\})$ to obtain the modal chain $\langle\Gamma_0,\ldots,\Gamma_l\rangle$ such that for each $i$, $\rho(i,m)\U \Gamma_i$ and $\Gamma_l$ is $\phi$-averse. This allows us to define the $(l,m+1)$-record $\rho': \dom\rho\cup([0,l]\times\{m+1\})$ by \[\rho'(i,j) = \begin{cases}\Gamma_i&j=m+1\\ \rho(i,j)&\text{else.}\end{cases}\]
	We then use \Cref{averse} applied to $\ur{\rho'}$ to construct $\Delta\in\realth{l+1}{m+1}$ such that $\ur{\rho'}\R \Delta$ and $\phi\notin\Delta$. This allows us to define the record $\tau':\dom\rho'\cup\{(l+1,m+1)\}$: \[\tau'(i,j)=\begin{cases}\Delta&(i,j)=(l+1,m+1)\\ \rho'(i,j)&\text{else}.\end{cases}\]
	As $\phi\notin\ur{\tau'}$, by induction, $\trmp{l}{m}, \tau'\nVdash\phi$, so $\trmp{l}{m},\rho\nVdash\square\phi$.
	
	$\phi=\exists x\phi(x)$. If $\exists x\phi(x)\in \ur{\rho}$, by Henkinness, $\phi(c)\in\ur{\rho}$ for some constant $c\in\sigma^\nu_{l,m}$. By induction, $\trmp{l}{m},\rho\Vdash\phi(c)$, so $\trmp{l}{m},\rho\Vdash \exists x\phi(x)$. If $\exists x\phi(x)\notin\ur{\rho}$, then by $\Exist$, for any $c$, we have $\phi(c)\notin \ur{\rho}$. By induction, for every $c\in \sigma^\nu_{l,m}$, $\trmp{l}{m},\rho\nVdash\phi(c)$. Since every element in $\DWC(\rho)$ is named by some constant, it must be that $\trmp{l}{m},\rho\nvdash\exists x\phi(x)$.
	
	$\phi=\forall x\phi(x)$. Suppose that $\forall x\phi(x)\in \ur{\rho}$ and $\rho\IC \rho'$, and say that $\rho'$ is an $(l,m')$-record. Since we have $\forall x\phi(x)\in\ur{\rho'}$, for any $c\in\sigma^\nu_{l,m'}$, $\phi(c)\in\ur{\rho'}$. By induction $\trmp{l}{m'}, \rho'\Vdash \phi(c)$ for each $c$. We then have $\trmp{l}{m}\Vdash\forall x\phi(x)$. Now suppose that $\forall x\phi(x)\notin\ur{\rho}$. We now observe that $(\ur{\rho},\{\phi(c_{l,m}^+)\})$ must be consistent, since if it were not, \Cref{gengen} would immediately tell us that $\ur{\rho}\vdash\forall x\phi(x)$. Apply \Cref{getdschain} to the top row of $\rho$ and this pair to obtain a modal chain $\langle \Gamma_0,\ldots,\Gamma_l\rangle$ such that for each $i$, $\rho(i,m)\U \Gamma_i$ and $\Gamma_i$ extends $(\ur{\rho},\{\phi(c_{l,m}^+)\})$. Using this modal chain, we define $\rho':\dom\rho\cup([0,l]\times\{m+1\})$ by
	\[\rho'(i,j) = \begin{cases} \Gamma_i& j=m+1\\ \rho(i,j) & \text{else}.\end{cases}\]
	Since $\phi(c_{l,m}^+)\notin\ur{\rho'}$, the induction hypothesis grants us $\trmp{l}{m+1},\rho'\nVdash\phi(c_{l,m}^+)$, which implies that $\trmp{l}{m+1},\rho'\nVdash\forall x\phi(x)$. By \Cref{partialtruth}, $\trmp{l}{m},\rho'\nVdash\forall x\phi(x)$, so as $\rho\IC\rho'$, $\trmp{l}{m},\rho\nVdash\forall x\phi(x)$.
\end{proof}

\begin{thm}[Strong completeness]
	If $\Gamma\Vdash\phi$, then $\Gamma\vdash \phi$.
\end{thm}
\begin{proof}
	In some signature $\sigma$, suppose that $\Gamma\nvdash\phi$. This is tantamount to the consistency of the pair $(\Gamma,\{\phi\})$. Therefore by \Cref{getreal}, in an expanded signature $\sigma'$, there is a saturated theory $\Delta$ such that $\Gamma\subseteq\Delta$ and $\phi\notin\Delta$. We look at the trace model $\trm$ for $\sigma'$. There is a $(0,0)$-record $\rho$ that is just the single point $(0,0)$ mapped to $\Delta$. By the truth lemma, $\trm,\rho\Vdash\Gamma$ but $\trm\nVdash\phi$, so $\Gamma\nVdash\phi$. 
\end{proof}

\section{Other First-Order Fischer Servi Logics}
\label{otherlogics}
\subsection{Extensions by Modal Axioms}
\label{modext}
In this section we consider some extensions of $\fofs$ for which the techniques in this paper directly carry over to establish completeness. We begin by introducing some familiar axioms: $\square \phi\to\lozenge\phi$ ($\dboth$), $\square\phi\to\phi$ ($\tb$), $\phi\to\lozenge\phi$ ($\td$), $\square\phi\to\square\square\phi$ ($\fourbox$), and $\lozenge\lozenge\phi\to\lozenge\phi$ ($\fourdiamond$). These then allow us to define the five additional first-order Fischer Servi logics in \Cref{morelogics}.
\begin{figure}
	\begin{tabular}{| l | l | l | l : l |}
		\hline
		\multicolumn{5}{| c |}{Extensions of $\fofs$}\\
		\hline
		Logic & Definition &Frame Condition & Frame Class & (abbr.)\\
		\hline
		$\fofsd$ & $\fofs\oplus\{\dboth\}$ & $\M$ is serial&$\class[\fofsd]$&$\class[\fofsda]$\\
		$\fofsfour$ & $\fofs\oplus\{\fourbox,\fourdiamond\}$ & $\M$ is transitive&$\class[\fofsfour]$&$\class[\fofsfoura]$\\
		$\fofsdfour$ & $\fofs\oplus\{\dboth,\fourbox,\fourdiamond\}$ & $\M$ is serial and transitive&$\class[\fofsdfour]$&$\class[\fofsdfoura]$\\
		$\fofst$ & $\fofs\oplus\{\tb,\td\}$ & $\M$ is reflexive&$\class[\fofst]$&$\class[\fofsta]$\\
		$\fofssfour$ & $\fofs\oplus\{\tb,\td,\fourbox,\fourdiamond\}$ & $\M$ is a quasi-order&$\class[\fofssfour]$&$\class[\fofssfoura]$\\
		\hline
	\end{tabular}
\caption{The logics considered in \Cref{modext} and their corresponding frame classes}
\label{morelogics}
\end{figure}

\begin{thm}
	For $\logic\in \{\fofsfour,\fofsd,\fofsdfour,\fofst,\fofssfour\}$, $\logic$ is sound with respect to $\class[\logic]$.
\end{thm}
\begin{proof}
	In light of \Cref{sound}, establishing each of these additional soundness results simply requires verifying that in each case the frame condition forces the truth of the new axiom or axioms. Each of these cases is covered in \cite{amaandpir}.
\end{proof}

We can also obtain a completeness result for each of these logics. In contrast to canonical model constructions, even the addition of a simple axiom may require that we alter the definition of modal accessibility to obtain the correct trace model.

First we consider $\fofsd$. In this case we do not have to manually change any part of the definition of the trace model for the seriality condition to obtain. Simply note that for any saturated $\Gamma$, $(\yb{\Gamma},\nd{\Gamma})$ is consistent: if it were not, we would have $\vdash_\fofsda\phi\to\psi$ where $\phi\in\yb{\Gamma}$ and $\psi\in\nd{\Gamma}$. By regularity, $\vdash_\fofsda \lozenge\phi\to\lozenge\psi$, and by $\dboth$, $\lozenge\phi\in \Gamma$. Therefore $\lozenge\psi\in \Gamma$, which is a contradiction. This allows us to construct a modal successor for any record, so the modal relation in the trace model is serial. Hence $\trf_\fofsda$ is in the class $\class[\fofsda]$, so the truth lemma yields a completeness result.

\begin{thm}
	$\fofsd$ is complete with respect to $\class[\fofsda]$.
\end{thm}

For the remaining logics in this subsection, we will need to slightly adjust our trace frame construction. Namely, the definition of the modal relation, which is (\ref{tmodaldef}) in \Cref{tdef}, will be modified in order to guarantee transitivity and/or reflexivity. First, we consider $\fofsfour$ and $\fofsdfour$. For the trace frames $\trf_\fofsfoura$ and $\trf_\fofsdfoura$, we replace the definition of $\MC$ with the following:
\begin{center}
\begin{tabular}{c c c}
	(iii-$\mathsf{4}$) & & \Centerstack{Where $\rho$ is an $(l,m)$-record and $\rho'$ is an $(l',m')$-record, \\$\rho\MC \rho'$ just in case $l\neq l'$, $m=m'$, and $\rho'$ extends $\rho$.}
\end{tabular}
\end{center}
Evidently this is a transitive relation, and in the case of $\fofsdfour$, as was already established, it is moreover serial. Verification of the truth lemma is largely the same with only the right-to-left direction of the $\square$ case and the left-to-right direction of the $\lozenge$ case requiring a modified argument. For the $\square$ case, suppose $\square \phi\in \ur{\rho}$ and $\rho\IC \rho'\MC \tau$. Then $\square^n\phi\in \ur{\rho'}$ for any $n>0$ by $\fourbox$, so the definition of record ensures that $\phi\in\ur{\tau}$. For the $\lozenge$ case, if $\lozenge\phi\notin \ur{\rho}$, then $\lozenge^n\phi\notin\ur{\rho}$ for any $n>0$ by $\fourdiamond$. The definition of record then ensures that $\phi\notin\ur{\tau}$ for any $\tau$ with $\rho\MC \tau$.

\begin{thm}
	$\fofsfour$ is complete with respect to $\class[\fofsfoura]$ and $\fofsdfour$ is sound complete with respect to $\class[\fofsdfoura]$.
\end{thm}

For $\fofst$, we adjust the definition of the relation in the trace frame thus:
\begin{center}
	\begin{tabular}{c c c}
		(iii-$\mathsf{T}$) & & \Centerstack{Where $\rho$ is an $(l,m)$-record and $\rho'$ is an $(l',m')$-record, \\$\rho\MC \rho'$ just in case $l=l'-1$ or $l=l'$, $m=m'$, and $\rho'$ extends $\rho$.}
	\end{tabular}
\end{center}

Since trivially any record extends itself, $\MC$ will indeed be reflexive. The same two cases from the truth lemma will again require modified arguments. If $\square\phi\in \rho$ and $\rho\IC \rho'\MC \tau$, we now have the additional possibility that $\rho'=\tau$. But by $\tb$, $\phi\in\ur{\rho'}$. Similarly if $\lozenge\phi\notin \rho$ and $\rho\MC \tau$, we have to consider that $\rho=\tau$. By $\td$ we cannot have $\phi\in \ur{\rho}$, however, so we still have that for all $\tau$ with $\rho\MC \tau$, $\phi\notin \ur{\tau}$.

\begin{thm}
	$\fofst$ is complete with respect to $\class[\fofsta]$.
\end{thm}

Finally we consider $\fofssfour$. The adjusted definition of the modal relation follows:
\begin{center}
	\begin{tabular}{c c c}
		(iii-$\mathsf{S4}$) & & \Centerstack{Where $\rho$ is an $(l,m)$-record and $\rho'$ is an $(l',m')$-record, \\$\rho\MC \rho'$ just in case $m=m'$, and $\rho'$ extends $\rho$.}
	\end{tabular}
\end{center}
This relation is clearly a quasi-order. (In fact, it is a partial order, since some $(l,m)$-record extends another just in case they are equal.) The additional arguments required here are all from either the $\fofsfour$ case or the $\fofst$ case.
\begin{thm}
	$\fofssfour$ is complete with respect to $\class[\fofssfoura]$.
\end{thm}
\begin{center}
	\begin{figure}
		\begin{tikzpicture}
			\node (fofs) at (0,0) {$\fofs$};
			\node (fofs4) at (-1.4,1.4) {$\fofsfour$};
			\node (fofsd) at (1.4,1.4) {$\fofsd$};
			\node (fofsd4) at (0,2.8) {$\fofsdfour$};
			\node (fofst) at (2.8,2.8) {$\fofst$};
			\node (fofss4) at (1.4,4.2) {$\fofssfour$};
			\draw (fofs) -- (fofs4);
			\draw (fofs) -- (fofsd);
			\draw (fofs4) -- (fofsd4);
			\draw (fofsd) -- (fofsd4);
			\draw (fofsd) -- (fofst);
			\draw (fofst) -- (fofss4);
			\draw (fofsd4) -- (fofss4);
		\end{tikzpicture}
		\caption{The lattice of first-order Fischer Servi logics considered in this paper}
		\label{extlat}
	\end{figure}
\end{center}
\subsection{Necessity of Identity and Necessity of Distinctness}
\label{idextensions}
Let $\logic$ be one of the six logics in \Cref{extlat}. We now consider extending $L$ by the necessity of identity and the necessity of distinctness. Formally, we have the axioms $c_1\equals c_2\to \square(c_1\equals c_2)$ ($\NI$) and $\lozenge(c_1\equals c_2)\to c_1\equals c_2$ ($\ND$).

Note that adding $\NI$ to $\logic$ is equivalent to lifting the modal-free restriction on $\idsub$. For one of the aforementioned frame classes $\class[\logic]$, we obtain the restricted class $\class[\logic]^\Rightarrow$ by imposing the condition that for any $w\M v$, if $d_1\DE[w] d_2$ then $d_1\DE[v] d_2$. Similarly, $\class[\logic]^\Leftarrow$ is the subclass of $\class[\logic]$ obtained by imposing the condition that for any $w\M v$, if $d_1\DE[v] d_2$ then $d_1\DE[w] d_2$. Finally, $\class[\logic]^\Leftrightarrow=\class[\logic]^\Rightarrow\cap\class[\logic]^\Leftarrow$. Soundness is easily checked and completeness can be proved using the trace model constructions found in \Cref{tracemod} and \Cref{modext}.

\begin{thm}
	For $\logic\in \{\fofs,\fofsfour,\fofsd,\fofsdfour,\fofst,\fofssfour\}$, we have the following:
	\begin{itemize}
		\item $\logic\oplus\{\NI\}$ is sound and complete with respect to $\class[\logic]^\Rightarrow$;
		\item $\logic\oplus\{\ND\}$ is sound and complete with respect to $\class[\logic]^\Leftarrow$;
		\item $\logic\oplus\{\NI,\ND\}$ is sound and complete with respect to $\class[\logic]^\Leftrightarrow$.
	\end{itemize}
\end{thm}

\section{Future Work}
The most notable extension of $\fs$ for which we could not treat a first-order analogue is $\m$. It is unclear whether a trace model-like construction would be apt for establishing completeness of first-order $\m$, as the trace model approach takes advantage of increasing domains, something that must be given up when the modal relation is required to be an equivalence relation. More nebulously, it seems likely that trace model constructions could find applications in establishing completeness in contexts where a significant amount of control over the behavior of one or more of the relations is needed. 

\bibliographystyle{tugboat}
\bibliography{fofs_bib}

\end{document}